%% file: FMZ20-comb-vol-04-30.tex
\documentclass[11pt,a4paper]{article}

\input{paper_defns}

\usepackage[T1]{fontenc}
\usepackage{epsfig}
\usepackage{mathtools}
\usepackage[applemac]{inputenc}		
\usepackage{hyperref}
\usepackage{comment}

\hypersetup{colorlinks=true,pdftitle=""}

\newtheorem{conj}{Conjecture}[section]
\newtheorem{thm}[conj]{Theorem}

\newtheorem{rem}[conj]{Remark}
\newtheorem{lem}[conj]{Lemma}
\newtheorem{prop}[conj]{Proposition}

\newtheorem{ques}[conj]{Question}

\newtheorem{defn}[conj]{Definition}
\newtheorem{cor}[conj]{Corollary}

\begin{document}

\title{Sumset estimates in convex geometry}
\author{Matthieu Fradelizi, 
Mokshay Madiman\thanks{supported in part by the U.S. National Science Foundation through grants DMS-1409504 (CAREER) and CCF-1346564.}, 

and Artem Zvavitch\thanks{ supported in part by the U.S. National Science
Foundation Grant DMS-1101636, Simons Foundation, CNRS and the B\'ezout Labex funded by ANR, reference ANR-10-LABX-58.}}

\date{\today}

\maketitle

\begin{abstract}
Sumset estimates, which provide bounds on the cardinality of sumsets of finite sets in a group, form an essential part of the toolkit of additive combinatorics. In recent years,  probabilistic or entropic analogs of many of these inequalities were introduced.  We study analogues of these sumset estimates in the context of convex geometry and Lebesgue measure on $\RL^n$. First, we observe that, with respect to Minkowski summation, volume is supermodular to arbitrary order on the space of convex bodies. Second, we explore sharp constants in the convex geometry analogues of  variants of the Pl\"unnecke-Ruzsa inequalities. In the last section of the paper, we provide connections of these inequalities to the classical Rogers-Shephard inequality.
\end{abstract}





\tableofcontents

\section{Introduction}
\label{sec:intro}

Minkowski summation is a basic and ubiquitous operation on sets. Indeed, the Minkowski
sum $A+B = \{a+b : a \in A, b \in B\}$ of sets $A$ and $B$ makes sense as long as $A$ and $B$ are subsets
of an ambient set in which a closed binary operation denoted by $+$ is defined. In particular, this notion makes sense
in any group, and {\it additive combinatorics} (which arose out of exploring the additive structure of sets of integers, but then expanded to the consideration of additive structure in more general groups) is a field of mathematics that is preoccupied with studying what exactly this operation does in a quantitative way. 

``Sumset estimates'' are a collection of inequalities developed in additive combinatorics that provide bounds on the cardinality of sumsets of finite sets in a group. In this paper, we use $\#(A)$ to denote the cardinality of a countable set $A$, and $|A|$ to denote the volume (i.e., $n$-dimensional Lebesgue measure) of $A$ when $A$ is a measurable subset of $\R^n$. The simplest sumset estimate is the two-sided inequality 
$\#(A) \#(B) \geq \#(A+B) \geq \#(A)+\#(B)-1$, which holds for finite subsets $A, B$ of the integers; equality in the second inequality holds only for arithmetic progressions.
A much more sophisticated sumset estimate is Kneser's theorem \cite{Kne53} (cf., \cite[Theorem 5.5]{TV06:book}, \cite{Dev14}), which asserts that for finite, nonempty subsets $A, B$ in any abelian group $G$,
$\#(A+B) \geq \#(A+H) +\#(B+H)-\#(H)$, where $H$ is the stabilizer of $A+B$, i.e., $H=\{g\in G: A+B+g=A+B\}$. Kneser's theorem contains, for example, the Cauchy-Davenport inequality that provides a sharp lower bound on sumset cardinality in $\mathbb{Z}/p\mathbb{Z}$. In the reverse direction of finding upper bounds on cardinality of sumsets, there are the so-called Pl\"unnecke-Ruzsa inequalities \cite{Plu70, Ruz89}. One example of the latter states that if $\#(A+B)\leq \alpha \# A$, then $\#(A+k\cdot B)\leq \alpha^k \# A$, where $k\cdot B$ refers to the sum of $k$ copies of $B$.
Such sumset estimates form an essential part of the toolkit of additive combinatorics.

In the context of the Euclidean space $\RL^n$, inequalities for the volume of Minkowski sums of convex sets, and more generally Borel sets, play a central role in geometry and functional analysis. For example, the well known Brunn-Minkowski inequality can be used to deduce the Euclidean isoperimetric inequality, which identifies the Euclidean ball as the set of any given volume with minimal surface area.
Therefore, it is somewhat surprising that in the literature, there has only been rather limited exploration of geometric analogies of sumset estimates. We work towards correcting that oversight in this contribution. 

The goal of this paper is to explore a variety of new inequalities for volumes of Minkowski sums of convex sets, which have a combinatorial flavor and are inspired by known inequalities in the discrete setting of additive combinatorics. These inequalities are related to the notion of supermodularity: we say that a set function $F:2^{[n]}\ra\RL$ is {\it supermodular} if 
$F(\setS\cup\setT)+F(\setS\cap\setT) \geq F(\setS) + F(\setT)$
for all subsets $\setS, \setT$ of $[n]$, and that $F$ is {\it submodular} if $-F$ is supermodular. 

Our study is motivated by two relatively recent observations.
The first observation motivating this paper, due to \cite{FMMZ18} (Theorem 4.5), states that given convex bodies $A, B_1, B_2$ in $\RL^{n}$, 
$|A+B_1+B_2|+|A|\geq |A+B_1|+|A+B_2|$.
This inequality has a form similar to that of Kneser's theorem-- indeed, observe that the latter can be written as
$\#(A+B+H)+\#(H) \geq \#(A+H) +\#(B+H)$, since adding the stabilizer to $A+B$ does not change it.
Furthermore, it implies that the function $v: 2^{[n]} \to \RL$ defined, for given convex bodies $B_1, \dots, B_k$ in $\RL^{n}$, by  
$v(\setS)=\left|\sum_{i\in \setS}  B_i \right|$
is supermodular. 
Foldes  and Hammer \cite{FH05} defined the notion of higher order supermodularity for set functions. In  Section~\ref{sec:super}, we generalize their definition and main characterization theorem from \cite{FH05} to functions defined on $\RL_+^n$, and apply it to show that volumes and mixed volumes satisfy this higher order supermodularity.

The second observation motivating this paper is due to Bobkov and the second named author  \cite{BM12:jfa}, who  proved that given convex bodies $A, B_1, B_2$ in $\RL^{n}$,  
\begin{equation}\label{eqBM3}
|A+B_1+B_2 | |A| \leq 3^n  |A+B_1| |A+B_2|.
\end{equation}
The above inequality is inspired by an inequality in information theory analogous to the Pl\"unnecke-Ruzsa inequality (the most general version of which was proved by Ruzsa for compact sets in \cite{Ruz97}, and is discussed in Section \ref{secPR} below).
 If not for the multiplicative factor of $3^n$ in (\ref{eqBM3}), this inequality would imply that the {\it logarithm} of the volume of the Minkowski sum of convex sets is submodular. In this sense, it goes in the reverse direction to the supermodularity of volume and thus complements it. However, the constant $3^n$ obtained by \cite{BM12:jfa} is rather loose. We take up the question of tightening this constant in Section~\ref{sec:Pl\"unnecke}. 
 
 Specifically, we obtain both upper and lower bounds for the optimal constant 
 \be\label{cn-def}
 c_n=\sup \frac{|A+B+C | |A| }{|A+B| |A+B|} ,
 \ee
 where the supremum is taken over all convex bodies $A, B, C$ in $\RL^n$,
 in general dimension $n$. We get an upper bound of $c_n\leq \varphi^{n}$ in Section~\ref{ss:gen-ub}, where $\varphi=(1+\sqrt{5})/2$ is the golden ratio, and an asymptotic lower bound of $c_n\geq (4/3+o(1))^n$ in Section~\ref{ss:PR-LB}. In Section~\ref{ss:ub34}, we show that the optimal constant is $1$ in dimension $2$ and $\frac{4}{3}$ in dimension 3 (i.e., $c_2=1$ and $c_3=4/3$), and also that $c_4\leq 2$. 
 In Section~\ref{ss:improved}, we improve inequality (\ref{eqBM3}) in the special case where $A$ is an ellipsoid, $B_1$ is a zonoid, and $B_2$ is any convex body: in this case, the optimal constant is $1$. This result partially answers a question of Courtade, who asked (motivated by an analogous inequality in information theory) if $|A+B_1+B_2|\,|A|\le |A+B_1|\,|A+B_2|$ holds when $A$ is the Euclidean ball and $B_1, B_2$ are arbitrary convex bodies. 
Finally, in Section~\ref{ss:compact}, we prove that  (\ref{eqBM3}) cannot possibly hold in the more general setting of compact sets with any absolute constant, which signifies a sharp difference between the proof of this inequality compared with the tools used by Ruzsa in  \cite{Ruz97}.

The last section of the paper is dedicated to  questions surrounding Ruzsa's triangle inequality: if $A, B,$ and $C$ are finite subsets of an abelian group, then
$ \#(A)\#(B-C)\leq \#(A-B)\#(A-C)$. The inequality is also known to be true for volume of compact sets in $\RL^n$: $|A|\,|B-C|\le |A-B|\,|A-C|$. We investigate the best constant $c$ such that the inequality 
\begin{equation}\label{triang3}
|A|\,|A+B+C|\leq c |A-B|\,|A-C|.
\end{equation}
is true for all convex sets $A,B$ and $C$ in $\RL^n$. For example, in the plane, we observe that it holds with the sharp constant $c=\frac{3}{2}$.
Again, it is interesting to note that (\ref{triang3}) is different from Ruzsa's triangle inequality, and it is not true, with any absolute constant $c$, if one omits the assumption of convexity.


In a companion paper \cite{FMMZ22}, we explore the question of reducing the constant in the Pl\"unnecke-Ruzsa inequality for volumes from $\varphi^{n}$, when we restrict attention to the subclass of convex bodies known as zonoids. In another companion paper \cite{FLMZ22}, we explore measure-theoretic extensions of the preceding results for convex bodies, in the category of $\log$-concave and in particular Gaussian measures.

We also mention that there are probabilistic or entropic analogs of many of the inequalities in this paper. For example, the afore-mentioned observation due to \cite{BM12:jfa}, that a Pl\"unnecke-Ruzsa inequality for convex bodies holds with a constant $3^n$, emerges as a consequence of R\'enyi entropy comparisons for convex measures on the one hand, and the submodularity of entropy of convolutions on the other. The submodularity of entropy of convolutions refers
to the inequality $h(X)+h(X+Y+Z)\leq h(X+Y)+h(X+Z)$, where $h$ denotes entropy, and $X, Y, Z$ are independent $\RL^n$-valued random variables, and may be thought of as an entropic analogue of the Pl\"unnecke-Ruzsa inequality. This latter inequality was obtained in \cite{Mad08:itw} as part of an attempt to develop an additive combinatorics of probability measures where cardinality or volume is replaced by entropy. A number of works have explored this avenue, starting with
\cite{Ruz09:1, Tao10, MMT12, ALM17, MWW21} for discrete probability measures on groups (e.g., when the random variables take values in finite groups or the integers),  and 
\cite{Mad08:itw, MK10:isit, MK18, Hoc22} for probability measures on $\RL^n$ and more general locally compact abelian groups.

\vspace{.1in}
\noindent
{\bf Acknowledgments.}
Piotr Nayar and Tomasz Tkocz \cite{NT17} independently obtained upper and lower bounds on the optimal constants in the Pl\"unnecke-Ruzsa inequality for volumes (versions of Theorems~\ref{thm:ub} and \ref{thm:lb}, though with weaker bounds obtained using different methods); we are grateful to them for communicating their work.
We are indebted to Ramon Van Handel for pointing to us the original work of W. Fenchel on the local version of Alexandrov's inequality, to Daniel Hug for suggesting that we consider equality cases in Theorem \ref{lmcool2}, and to  Mathieu Meyer and Dylan Langharst  for a number of  valuable discussions and suggestions.

%
%
%
%
%
%

\section{Preliminaries}
\label{sec:prelim}

\subsection{Mixed Volumes}

In this section, we introduce basic notation and collect essential facts and definition 
from Convex Geometry that are used in the paper. As a general reference on the theory we use \cite{Sch14:book}.
We write $x \cdot y$  for the inner product of vectors $x$ and $y$ in $\RL^n$  and  by $|x|$ the length of a vector $x \in \RL^n$. The closed unit ball in $\RL^n$ is denoted by $B_2^n$, and its boundary by $S^{n-1}$. We will also denote by $e_1, \dots, e_n$ the standard orthonormal basis in $\RL^n$. Moreover, for any set in $A \subset \RL^n$, we denote its boundary by $\partial A$. A convex body is a convex, compact set with nonempty interior. We write $|K|_m$ for the $m$-dimensional Lebesgue measure (volume) of a measurable set $K \subset \RL^n$, where $m = 1, . . . , n$ is the dimension
of the minimal affine space containing $K$, we will often use the  shorten notation  $|K|$ for $n$-dimensional volume.  A polytope which is the Minkowski sum of
finitely many line segments is called a zonotope. Limits of zonotopes in the Hausdorff
metric are called zonoids, see \cite{Sch14:book}, Section 3.2,  for details.
From  \cite[Theorem 5.1.6]{Sch14:book}, for any compact convex sets $K_1,\dots K_r$ in $\RL^n$ and any non-negative numbers $t_1, \dots, t_r$ 
one has
\be\label{eq:mvf}
\left|t_1K_1+\cdots+t_rK_r\right|=
\sum_{i_1,\dots,i_n=1}^rt_{i_1}\cdots t_{i_n}V(K_{i_1},\dots K_{i_n}),
\ee
for some non-negative numbers $V(K_{i_1},\dots K_{i_n})$, which are called the mixed volumes of $K_1,\dots K_r$.
One readily sees that the mixed volumes satisfy $V(K,\dots,K)=|K|$, moreover, they satisfy a number of properties which are crucial
for our study (see  \cite{Sch14:book}) including the fact that a mixed volume is   symmetric in its argument; it is multilinear, i.e. for any $\lambda, \mu \ge 0$ we have $V(\lambda K + \mu L, K_2,  \dots, K_n)=\lambda V(K,K_2,  \dots, K_n)+ \mu V(L,K_2, \dots, K_n).$ 
Mixed volume is translation invariant, i.e.  $V(K+a,K_2,  \dots K_n)= V(K,K_2, \dots, K_n),$ for $a \in \RL^n$ and satisfy a monotonicity property, i.e   $V(K,K_2, K_3, \dots, K_n) \le V(L,K_2, K_3, \dots, K_n)$, for $ K \subset L$.
We will also often use a two body version of (\ref{eq:mvf}) -- the Steiner formula:
\be\label{eq:ste}
\left|A+tB\right|=
\sum_{k=0}^n {n \choose k} t^k V(A[n-k],B[k]),
\ee
for any $t>0$ and compact, convex sets $A, B$ in $\RL^n$, where for simplicity we use notation $A[m]$ for a convex set $A$ repeated $m$ times.  Mixed volumes are also very useful for studying  the volume of orthogonal projections of convex bodies. Let $P_H A$ be an orthogonal projection of a convex body $A$ onto $m$ dimensional subspace $H$ of $\RL^n$, then 
\be\label{eq:proj}
|P_HA|_m|U|_{n-m}={{n}\choose{m}}V(A[m], U[n-m]),
\ee
where $U$ is any convex body of volume one in the subspace $H^\perp$ orthogonal to $H$. For example, if we denote by $\theta^\perp = \{x \in \RL^n : x \cdot \theta =0\}$ a hyperplane orthogonal to $\theta\in S^{n-1}$, we obtain
\be\label{eq:proj1}
|P_{\theta^\perp} A|_{n-1}= n V(A[n-1], [0, \theta]).
\ee
Yet another useful formula is connected with computation of surface area and mixed volumes:
\be\label{surface}
|\partial A|= nV(A[n-1], B_2^n),
\ee
where by $|\partial A|$ we denote the surface area of the compact set $A$ in $\RL^n$.
Mixed volumes satisfy a number of extremely useful inequalities. The first one is the Brunn-Minkowski inequality
\be\label{BM}
|A+B|^{1/n}\ge |A|^{1/n}+|B|^{1/n},
\ee
whenever  $A,B$ and $A+B$ are measurable.
The most powerful inequality for mixed volumes is the 
Alexandrov--Fenchel inequality: 
\be\label{AF}
V(K_1,K_2, K_3, \dots, K_n) \ge \sqrt{V(K_1,K_1, K_3, \dots, K_n)V(K_2,K_2, K_3, \dots, K_n)},
\ee
for any compact convex sets $K_1,\dots K_r$ in $\R^n$. 
We will also use the following classical  local version of Alexandrov-Fenchel's inequality that was proved by W. Fenchel (see \cite{Fen36} and also \cite{Sch14:book}) and further generalized in \cite{FGM03, AFO14, SZ16} 
\begin{equation}\label{alexloc}
|A|V(A[n-2], B, C)\le 2 V(A[n-1],B) V(A[n-1],C),
\end{equation}
 for any convex compact sets $A,B,C$ in $\RL^n$, moreover it was noticed in \cite{SZ16} that (\ref{alexloc}) is true with constant one instead of two in the case when $A$ is a simplex. The inequality turned out to be a part of rich class Bezout inequalities proposed in  \cite{ SZ16, SSZ16}. The core tool of our work is the following inequality of J.~Xiao (Theorem 1.1 and Lemma 3.3 in \cite{Xia19}) 
\begin{align}\label{eq:jxiao}
 |A|V(A[n-j-m], &B[j], C[m]) \nonumber \\&\le \min\left( {n \choose j} , {n \choose m} \right)V(A[n-j],B[j])V(A[n-m],C[m]). 
\end{align}

\subsection{Pl\"unnecke-Ruzsa inequality.}\label{secPR}

Pl\"unnecke-Ruzsa inequalities   (see for example \cite{TV06:book}) is an important class of inequalities in the field of 
additive combinatorics. These were introduced by Pl\"unnecke \cite{Plu70} and
generalized by Ruzsa \cite{Ruz89}, and a simpler proof was given by Petridis \cite{Pet12}; a more recent generalization is proved 
in \cite{GMR08}, and entropic versions are developed in \cite{MMT12}. 
For illustration, the form of Pl\"unnecke's inequality developed in  \cite{Ruz89} states that,
if $A, B_1,\ldots ,B_m$ are finite sets in a commutative group, then there exists an $X \subset A, X \neq \emptyset$, such that
\ben
\#(A)^m\#(X+B_1+\ldots +B_m) \leq \#(X)  \prod_{i=1}^m\#(A + B_i).
\een
In \cite{Ruz97}, Ruzsa generalized the above inequality to the case of compact sets on a locally compact commutative group, written additively, with the  Haar measure. The volume case of this deep theorem is one of our main inspirations: for any compact sets $A,B_1, \ldots ,B_m$ in $\R^n,$ with $|A|>0$ and for every $\varepsilon >0$ there exists a compact set $A' \subset A $ such that 
\begin{equation}\label{eq:ruzvol}
|A|^m|A'+B_1+\ldots +B_m| \le (1+\varepsilon)|A'| \prod_{i=1}^m|A + B_i|.
\end{equation}
It immediately follows that  for any compact sets $A,B_1, \ldots ,B_m$ in $\R^n,$
\begin{equation}\label{eq:ruzvol1}
|A|^{m-1}|B_1+\ldots +B_m| \le \prod_{i=1}^m|A + B_i|.
\end{equation}

\subsection{Submodularity and supermodularity}
\label{sec:smod-prelim}

Let us first recall the notion of a supermodular set function. 

\begin{defn}\label{def:supermod}
A set function $F:2^{[n]}\ra\RL$ is {\it supermodular} if 
\be\label{supmod:defn}
F(\setS\cup\setT)+F(\setS\cap\setT) \geq F(\setS) + F(\setT)
\ee
for all subsets $\setS, \setT$ of $[n]$.
\end{defn}
One says that a set function $F$ is submodular if $-F$ is supermodular. 
Submodularity is closely related to a partial ordering on hypergraphs as we will see below. This relationship is frequently attributed to Bollobas and Leader \cite{BL91} (cf. \cite{BB12}), where they introduced the related notion of ``compressions''.
However, it seems to go back much longer -- it is more or less explicitly discussed in a 1975 paper of Emerson \cite{Eme75}, where he says it is ``well known''.

To present this relationship, let us introduce  some notation. Let $\calM(n,m)$ be the following family of 
(multi)hypergraphs: each consists of non-empty (ordinary) subsets $\setS_i$ of $[n]$,
$\setS_i=\setS_j$ is allowed, and 
$\sum_i |\setS_i|= m$. 
Consider a given multiset 
$\collS = \{\setS_1, \dots, \setS_l\} \in \calM(n,m)$.
The idea is to consider an operation that takes two sets in $\collS$
and replaces them by their union and intersection; however, note that
(i) if $\setS_i$ and $\setS_j$ are nested (i.e., either $\setS_i \subset \setS_j$ or 
$\setS_j \subset \setS_i$), then replacing $(\setS_i,\setS_j)$ by 
$(\setS_i \cap \setS_j ,\setS_i \cup \setS_j)$ does not change $\collS$, 
and
(ii) if  $\setS_i \cap \setS_j = \emptyset$, the empty set may enter the collection,
which would be undesirable.
Thus, take any pair of non-nested sets $\{\setS_i,\setS_j\}\subset \collS$
and let $\collS' = \collS(i,j)$ be obtained from $\collS$ by replacing $\setS_i$ and $\setS_j$ 
by $\setS_i \cap \setS_j$ and $\setS_i \cup \setS_j$, 
keeping only $\setS_i \cup \setS_j$ if $\setS_i \cap \setS_j = \emptyset$. 
$\collS'$ is called an {\it elementary compression} of $\collS$. The result of a sequence of
elementary compressions is called a {\it compression}.

Define a partial order on $\calM(n,m)$ by setting $\calA > \calA'$ if $\calA'$ is a
compression of $\calA$. To check that this is indeed a partial order, 
one needs to rule out the possibility of cycles, which can be done by noting that if $\calA'$ is an elementary compression of $\calA$ then 
\ben
\sum_{\setS\in\calA} |\setS|^2 < \sum_{\setS\in\calA'} |\setS|^2 .
\een

\begin{thm}\label{thm:compr}
Suppose $F$ is a supermodular function on the ground set $[n]$.
Let $\calA$ and $\calB$ be finite multisets of subsets of $[n]$, with $\calA > \calB$.
Then
\ben
\sum_{\setS\in\calA} F(\setS) \leq \sum_{\setT\in\calB} F(\setT) .
\een
\end{thm}

\begin{proof}
When $\calB$ is an {\it elementary} compression of $\calA$, the statement is immediate 
by definition, and transitivity of the partial order gives the full statement. 
\end{proof}

Note that for every multiset $\calA \in \calM(n,m)$ there is a unique minimal multiset
$\calA^{\#}$ dominated by $\calA$, i.e. $\calA^{\#}<\calA,$  consisting of the sets
$\setS^{\#}_j = \{i \in [n] : i \text{ lies in at least } j \text{ of the sets } \setS \in \calA\}$.
Thus a particularly nice instance of Theorem~\ref{thm:compr} is for the special case of $\calB=\calA^{\#}$ (we refer to  \cite{BB12} page 132 for further discussion).
We also have a notion of supermodularity on the positive orthant of the Euclidean space.

\begin{defn}\label{def:supermod-Rmaxmin}
A function $f:\RL_{+}^{n} \ra\RL$ is supermodular if
\ben
 f(x\vee y) + f(x \wedge y) \ge f(x) + f(y) 
\een
for any $x,y \in \mathbb{R}_+^n$, where $x \vee y$ denotes the componentwise maximum of $x$ and 
$y$ and $x \wedge y$ denotes the componentwise 
minimum of $x$ and $y$.
\end{defn}
We note that Definition \ref{def:supermod-Rmaxmin} can be viewed as an extension of Definition~\ref{def:supermod} if one consider set functions on $2^{[n]}$ as a function on $\{0;1\}^n$. Indeed, if $f:\RL_{+}^{n} \ra\RL$ is supermodular then we define the function $F:2^{[n]}\to\RL$, for $s\subset[n]$, by $F(s)=f(e(s))$, where $e(s)_i=1$ if $i\in s$ and $e(s)_i=0$ if $i\notin s$. Then, the set function $F$ is supermodular:


\begin{lem}\label{lem:set-rn}
If $f:\RL_{+}^n\ra\RL$ is supermodular, and we set $F(s):=f(e(s))$ for each $s\subset [n]$,
then $F$ is a supermodular set function.
\end{lem}

\begin{proof}
Observe that
\ben\begin{split}
F(s\cup t) + F(s\cap t) 
&= f(e(s\cup t)) + f(e(s\cap t)) \\
&= f(e(s) \vee e(t)) + f(e(s)\wedge e(t)) \\
&\geq f(e(s))+f(e(t)) \\
&= F(s) + F(t) .
\end{split}\een
\end{proof}
The fact that supermodular functions are closely related to functions 
with increasing differences is classical (see, e.g., \cite{MG19} or  \cite{Top98:book}, which describes
more general results involving arbitrary lattices). 
We will denote by $\partial_i f$ the partial derivative of function $f$ with the respect to the $i$'s coordinate and by $\partial^m_{i_1, \dots, i_m}$ the mixed derivative with respect to coordinates $i_1, \dots i_m$.

 \begin{prop}\label{prop:diff}
 Suppose a function $f:\RL_+^n\ra\RL$ is in $C^2$, i.e., it is twice-differentiable with
 a continuous Hessian matrix. Then $f$ is supermodular if and only if
 \ben
 \partial_{i,j}^2 f(x)\geq 0
 \een
 for every distinct $i, j \in [n]$, and for any $x\in \RL_+^n$.
 \end{prop}

 We will prove Proposition \ref{prop:diff} as part of a more general statement on the mixed derivatives of the supermodular functions of higher order (Theorem \ref{thm:diff} below).

\section{Higher order supermodularity of mixed volumes}
\label{sec:super}

\subsection{Local characterization of higher order supermodularity}

We now present analogues of the above development for higher-order supermodularity.
Let us notice that a set function $F:2^{[n]}\to\RL$ is supermodular if and only if for any $s_0,s_1, s_2\in 2^{[n]}$ with $s_1\cap s_2=\emptyset$ one has 
\[
F(s_0\cup s_1)+F(s_0\cup s_2)\le F(s_0\cup s_1\cup s_2)+F(s_0).
\]
Generalizing this property, Foldes  and Hammer \cite{FH05} defined the notion of higher order supermodularity. In this section, we will adapt their definition and  study the following property:
\begin{defn}\label{eq:hsup}
Let $1\le m\le n$. A function $F:2^{[n]}\to\RL$ is $m$-supermodular if for any $s_0\in 2^{[n]}$ and for any mutually disjoint $s_1,\dots, s_m\in2^{[n]}$ one has 
\[
\sum_{I\subset[m]}(-1)^{m-|I|}F\left(s_0\cup\bigcup_{i\in I}s_i\right)\ge0.
\]
\end{defn}
Note that for $m=2$ in the above definition, we recover a supermodular set function.
We also introduce the notion of higher-order supermodularity for functions defined on the positive orthant of a Euclidean space.
\begin{defn}\label{def:supermod-Rm}
Let $1\le m\le n$. A function $f:\RL_{+}^{n} \ra\RL$ is $m$-supermodular if
\ben
\sum_{I\subset[m]}(-1)^{m-|I|}f\left(x_0\vee\bigvee_{i\in I}x_i\right)\ge0,
\een
for any $x_0$ and any $x_1,\dots,x_m \in \mathbb{R}_+^n$, with mutually disjoint supports, that is such that $x_i\wedge x_j=0$, for any $1\le i<j\le m$.
\end{defn}

\begin{rem}\label{rk:setsuper}
Notice that, as in Lemma \ref{lem:set-rn}, if $f:\RL_+^n\to\RL$ is $m$-supermodular then $F:2^{[n]}\to\RL$ defined by $F(s)=f(e(s))$ is $m$-supermodular. 
\end{rem}

For $m=1$ in the above definition, we obtain that $f$ is $1$-supermodular if and only if it is non-decreasing in each coordinate.
For $m=2$, we recover a supermodular function on the orthant as we prove in the following lemma.

\begin{lem}
Let $f:\RL_+^n\ra\RL$. Then $f$ is supermodular if and only if for any $x,y,z\in\RL_+^n$ such that $y\wedge z=0$ one has 
\be\label{eq:disj}
f(x\vee y\vee z)+f(x)\ge f(x\vee y)+f(x\vee z).
\ee
\end{lem}
\begin{proof}
Suppose $f$ is supermodular and one has $x,y,z\in\RL_+^n$ such that $y\wedge z=0$. Then, we set $a=x\vee y$ and $b=x\vee z,$ then, $a\vee b=x\vee y\vee z$ and $a\wedge b=x$, since $y\wedge z=0$. Thus, 
\[
f(x\vee y\vee z)+f(x)- f(x\vee y)-f(x\vee z)=f(a\wedge b)+ f(a\vee b)-f(a)-f(b)\ge0.
\]
Now assume that $f$ satisfies (\ref{eq:disj}) and let $a,b\in\RL_+^n$. We set $x=a\wedge b$ and we define $y$ by putting $y_i=a_i$ for $i$ such that $b_i<a_i$ and $y_i=0$ otherwise. In the same way, we set $z_i=b_i$, for $i$ such that $a_i<b_i$ and $z_i=0$ otherwise. Then $x\vee y=a$, $x\vee z=b$ and $x\vee y\vee z=a\vee b$, hence, we conclude similarly.
\end{proof}

The next theorem generalizes Proposition \ref{prop:diff} to higher order supermodularity. 
\begin{thm}\label{thm:diff}
Let $f:\RL_+^n\ra\RL$ be a $C^m$ function. Then $f$ is $m$-supermodular if and only if
\ben
\partial^m_{i_1, \dots, i_m} f(x)\geq 0
\een
for every distinct $i_1,\dots, i_m \in [n]$, and for any $x\in \RL_+^n$.
\end{thm}

\begin{proof}
Let $x_0\in \mathbb{R}_+^n$ and $x_1,\dots,x_m \in \mathbb{R}_+^n$, with mutually disjoint supports.
\begin{equation}\label{eq:induction}
\begin{split}
\sum_{I\subset[m]}&(-1)^{m-|I|}f\left(x_0\vee\bigvee_{i\in I}x_i\right)\\
=&\sum_{I\subset[m-1]}(-1)^{m-|I|-1}f\left(x_0\vee\bigvee_{i\in I}x_i \vee x_m\right)
+\sum_{I\subset[m-1]}(-1)^{m-|I|}f\left(x_0\vee\bigvee_{i\in I}x_i\right)
\\
= & \sum_{I\subset[m-1]}(-1)^{m-1-|I|}\left[f\left(x_0\vee\bigvee_{i\in I}x_i \vee x_m\right) -f\left(x_0\vee\bigvee_{i\in I}x_i\right) \right]\\
=& \sum_{I\subset[m-1]}(-1)^{m-1-|I|}g_{x_m}\left(x_0\vee\bigvee_{i\in I}x_i ) \right), 
\end{split}
\end{equation}
where $g_{z}(x)=f(x \vee z)-f(x),$ for any $x,z\in \RL_+^n$.  Thus $f$ is $m$-supermodular if and only if $x \mapsto g_{z}(x)$ is $(m-1)$-supermodular for any $z\in\RL_+^n$ as a function on the coordinate subspace $H_z=\{x \in \RL_+^n ; x_i z_i=0, \forall\ i =1, \dots, n\}$. 

Now we are ready to prove the theorem  for the case $m=2$. In this case, the above equivalence (\ref{eq:induction}) gives us that $f$ is supermodular if and only if $x\mapsto g_z(x)$ is $1$-supermodular for any $z\in\RL_+^n$ as a function of $x\in H_z$ and thus $g_z$ is non-decreasing in each coordinate direction of $H_z$, i.e. for each coordinate index $i$ such that $z_i=0$. Thus, assuming differentiability, this is equivalent $\partial_i g_z \ge 0 $ for all $i$ which is a coordinate direction in $H_z$. Taking $z=z_je_j$, $z_j > 0$, we get $\partial_i g_{z_je_j} \ge 0$ for all $i\not=j$. Thus $\partial_i f (x \vee z_je_j)-\partial_if(x) \ge 0,$ and finally $\partial^2_{i,j} f(x) \ge 0$. Reciprocally, assuming that $\partial^2_{i,j} f(x) \ge 0$ for all $i\neq j$ and all $x$, we get $D_y \partial_i f (x) \ge 0$ for  $y\in \RL_+^n$ such that $y_i=0$, where  by $D_y f= y\cdot \nabla f $ we denote the directional derivative with respect to vector $y\in \RL^n$. Thus $\partial_i f(x+y)- \partial_i f(x) \ge 0$ for all $y\in \RL_+^n$ with $y_i=0$. Thus, considering $y\in \RL^n_+$, such that  $y_i=0$ and $y_j=x_j\vee z_j - x_j,$ $j\not=i$, we get $\partial_ig(z)=\partial_if(x \vee z)-\partial_if(x) \ge 0$ for all $i$ not in the support of $z$. 

We will finish the proof applying an  induction argument. Assume that the statement of the theorem is true for $m-1$, for some $m\ge 3$. Let  $f:\RL_+^n\ra\RL$ be a $C^m$  $m$-supermodular function. Then applying (\ref{eq:induction})  we get that $x\mapsto g_z(x)$ is $m-1$-supermodular for any $z\in\RL_+^n$ as a function of $x\in H_z$, which, applying inductive assumption, gives us 
\ben
\partial_{i_1}\cdots  \partial_{i_{m-1}} [f(x \vee z)-f(x)]\geq 0
\een
for every distinct $i_1,\dots, i_{m-1}$, coordinates of  $H_z$, applying it to $z=z_{i_m} e_{i_m}$ we get
\ben
\partial_{i_1}\cdots  \partial_{i_{m}} f(x)\geq 0.
\een
Now assume the partial derivative  condition of the theorem. Then for every $z\in \RL_+^n$ and $i_1,\dots, i_{m-1}$, coordinates of  $H_z$, we have 
\ben
\partial_{i_1}\cdots  \partial_{i_{m-1}} g_z(x)=\partial_{i_1}\cdots  \partial_{i_{m-1}} f(x \vee z)- \partial_{i_1}\cdots  \partial_{i_{m-1}}f(x),
\een
but $\partial_{i_1}\cdots \partial_{i_{m-1}}\partial_{i_{m}}f(x) \ge 0$ for every $i_m\not= i_k,$ for $k=1, \dots, m-1$ and thus for every $i_m$ for which $z_{i_m}\not = 0$. So $\partial_{i_1}\cdots \partial_{i_{m-1}}f(x)$ is an non-decreasing function in each coordinate $i_m$ for which $z_{i_m}\not = 0$:
$$
\partial_{i_1}\cdots  \partial_{i_{m-1}} f(x \vee z)- \partial_{i_1}\cdots  \partial_{i_{m-1}}f(x) \ge 0
$$
and, applying inductive assumption, we get that  $x\mapsto g_z(x)$ is $(m-1)$-supermodular for any $z\in\RL_+^n$ as a function of $x\in H_z$, which, finishes the proof with the help of  (\ref{eq:induction}).
\end{proof}

As an example, which will help to understand the connection of supermodularity to Minkowski sum of sets, let $\varphi: \RL_+\to\RL$ be a convex function. Then for every $a_0,a_1,a_2\in\RL_+$ one has 
\ben
\varphi(a_0+a_1)+\varphi(a_0+a_2)\le \varphi(a_0+a_1+a_2)+\varphi(a_0).
\een
This property can be seen again as the supermodularity of the function $\Phi:2^{[2]}\to\RL$ defined by $\Phi(s)=\varphi(a_0+e(s)_1a_1+e(s)_2a_2)$, for any $s\in 2^{[2]}$, where $e(s)$ is defined above Lemma~\ref{lem:set-rn}.

We remark in passing that the positivity of mixed partial derivatives and its global manifestation also arises in the theory of copulas in probability (see, e.g., \cite{Car09}). In particular, it is well known there that for smooth functions $C:[0,1]^m\ra [0,1]$, the condition $\partial^{m}_{1,2,\ldots,m} C\geq 0$
is equivalent to the condition that
$\sum_{z\in \{x_i, y_i\}^m} (-1)^{N(z)} C(z)\geq 0$ for every box $\prod_{i=1}^m [x_i,y_i] \subset [0,1]^m$, where 
$N(z)=\#\{k:z_k=x_k\}$.




\subsection{Higher order supermodularity of volume}
\label{ss:vol-super}




\begin{thm}\label{thm:supmodvol}
Let $n,k\in\Nat$. Let $B_1, \dots, B_k$ be convex compact sets of $\RL^{n}$. Then the function $v: \RL^k_+ \to \RL$ defined as 
\begin{equation}\label{eq:upmodvol}
v(x)=\left|\sum_{i=1}^k x_i B_i \right|
\end{equation}
is $m$-supermodular for any $m\in\N$. 
\end{thm}
\begin{proof}
From the mixed volume formula \eqref{eq:mvf}, the function $v$ is a polynomial with non-negative coefficients so its mixed derivatives of any order are non-negative on $\RL_+^n$. By Theorem \ref{thm:diff}, we conclude that it is $m$-supermodular for any $m$.
\end{proof}
\begin{rem}\label{rk:mixed-vol-sm}  Theorem \ref{thm:supmodvol} can be given in a more general form: for any natural number $l \le n$, any convex bodies $C_1,\dots, C_l$ in $\RL^n$ and any convex sets $B_1, \dots, B_k$ in $\RL^{n}$, the function $v: \RL^k_+ \to \RL$
$$v(x)=V\left(\left(\sum_{i=1}^k x_i B_i\right)[n-l], C_1, \dots, C_l \right)$$ is $m$-supermodular for any $m\in\N$. 
\end{rem}


We notice that in Theorem \ref{thm:supmodvol} the convexity assumption is essential. Indeed, as was observed in \cite{FMMZ18}, for $k=3$, there exists non convex sets $B_1, B_2, B_3$ such that  the function $v$ defined above is not supermodular. We will discuss this issue in more details in Section \ref{sec:notconv} below.
Using Theorem \ref{thm:supmodvol}, Remark \ref{rk:mixed-vol-sm}, Remark \ref{rk:setsuper} and Theorem \ref{thm:compr} we deduce the following corollary.

\begin{cor}\label{cor:supmod}
Let $n,k\in\Nat$ and $B_1,\dots B_k$ be compact convex sets of $\RL^n$. Let $0\le l \le n$ and $C_1,\dots, C_l$ be convex bodies in $\RL^n$. Then
\begin{enumerate}
\item the function $\bar{v}:2^{[k]}\ra [0,\infty)$ defined by
\be\label{def:setfn-v}
\bar{v}(\setS)=V\left(\left(\sum_{i\in\setS} B_i\right) [n-l], C_1, \dots, C_l \right) 
\ee
for each $\setS\subset [k]$, is a $m$-supermodular set function, for any $m\in\Nat$.
\item Let $\calA$ and $\calB$ be finite multisets of subsets of $[n]$, with $\calA > \calB$.
Then
\be
\sum_{\setS\in\calA} \bar{v}(\setS) \leq \sum_{\setT\in\calB} \bar{v}(\setT).
\ee
\end{enumerate}
\end{cor}

Let us note that the above $m$-supermodularity of the function $\bar{v}$ is equivalent to the fact that for any convex bodies $B_0, B_1, \dots, B_{k}, C_1, \dots, C_l$ in $\RL^n$ 
$$
\sum_{\setS\subset[k]}(-1)^{k-|\setS|}V\left(\left(B_0+\sum_{i\in \setS}B_i\right)[n-l], C_1, \dots, C_l\right)\ge0.
$$
Applying the previous theorem to $l=0$, we get
\begin{equation}\label{eq_B_0}
\sum_{\setS\subset[k]}(-1)^{k-|\setS|}\left|B_0+\sum_{i\in\setS} B_i\right| \geq 0.
\end{equation}
The above inequality for $k=n$ and $B_0=\{0\}$ follows also directly from the following classical formula (see Lemma 5.1.4 in \cite{Sch14:book})
$$
\sum_{\setS\subset[n]}(-1)^{n-|\setS|}\left|\sum_{i\in\setS} B_i\right| = n!V(B_1, \dots, B_n).
$$
In the same way, we can also give another proof of the general case of (\ref{eq_B_0}).
\begin{thm}\label{sch} 
Let $B_0, B_1, \dots, B_{m}$ be convex bodies in $\RL^n$.  
\ben
\sum_{\setS\subset[m]} (-1)^{m-|\setS|} \left|B_0+\sum_{i\in\setS} B_i\right| =\!\!\!
\displaystyle\sum\limits_{\substack{\sum_{i=0}^m k_i=n;\\ k_1, \dots, k_m \ge 1}} {n \choose k_0, k_1, \dots, k_m} V(B_0[k_0], B_1[k_1], \dots, B_m[k_m])
\een
for $m\le n$ and zero otherwise. 
\end{thm}
\begin{proof} Following  the proof of Lemma 5.1.4 in \cite{Sch14:book}. Define
$$
g(t_0, t_1, \dots, t_m)=\sum_{\setS\subset[m]} (-1)^{m-|\setS|} \left|t_0B_0+\sum_{i\in\setS} t_iB_i\right|.
$$
Observe that $g$ is a homogeneous polynomial of degree $n$ and note that $g(t_0, 0, t_2, \dots, t_m)=0$, which can be seen by noticing that, in this particular case, the sum is telescopic. This implies that, in the polynomial $g(t_0, \dots, t_m)$, all  monomials with non-zero coefficients must contain a non-zero power of $t_1$. The same being true for each $t_i$, $i \ge 1$,  there is no non-zero monomials if $m>n$.  If $m \le n$ all non-zero monomials must come from the case $|\setS|=m$, i.e. from
$$
|t_0B_0+t_1 B_1+ \dots +t_mB_m|,
$$
which finishes the proof.
\end{proof}






 Thanks to the fact that supermodular set functions taking the value 0 at the empty set are fractionally
superadditive (see, e.g., \cite{MP82, MT10}), we can immediately deduce the following inequality.
Let $n\ge1$, $k\ge2$ be integers and let $A_1, \dots, A_k$ be $k$ convex sets in $\R^n$. Then, for any
fractional partition $\beta$ using a hypergraph $\collS$  on $[k]$, 
\begin{eqnarray}\label{eq:fsa}
\left|\sum_{i=1}^kA_i\right| \ge \sum_{s\in\collS} \beta(s)\left|\sum_{j\in s}A_j\right|.
\end{eqnarray}
It was shown  in  \cite{BM21}
that \eqref{eq:fsa} actually extends to all compact sets in $\RL^n$, but supermodularity does not extend to compact sets as discussed in the next section.

\subsection{Going beyond convex bodies}\label{sec:notconv}

Consider sets $A, B \subset \R^n$, such that $0 \in B$. Define $\Delta_B(A)=(A+B)\setminus A$, and note that $A+B$ is always a superset of $A$ because of the assumption
that $0\in B$. 
The supermodularity of volume is also saying something about set increments. Indeed, for any sets $A, B, C$ consider
\ben
\Delta_C\Delta_B(A)
=\Delta_C \big((A+B)\setminus A\big)
=\bigg(\big((A+B)\setminus A\big)+C \bigg)\setminus \big((A+B)\setminus A\big).
\een
We have, if $0 \in B \cap C$:
\begin{align}\label{difference}
\left|\Delta_C\Delta_B(A)\right|
&=\left|\big((A+B)\setminus A\big)+C\right|- \left|(A+B)\setminus A\right| \nonumber\\
&\geq \left|(A+B+C)\setminus (A+C)\right|- \left|A+B\right|+\left|A \right|\nonumber\\
&= |A+B+C|- |A+C|- |A+B|+ |A|,
\end{align}
where 
the inequality follows   from the general fact that
$(K+C)\setminus (L+C)\subset (K\setminus L)+C$. Moreover, if $A,B,C$ are convex, compact sets then the estimate is  non-trivial, i.e., using Theorem~\ref{sch}  we get that
the right hand side of the above quantity is non-negative. 
It is interesting to note that the $\Delta$ operation is not commutative, i.e. $\Delta_C\Delta_B(A) \not = \Delta_A\Delta_C(A)$; this can be seen, for example, in $\RL^2$ by taking $A$ to be a square, $B$ to be a segment, and $C$ to be a Euclidean ball.

It is natural to ask if the higher-order analog of this observation remains true.
\begin{ques}
Let  $m\in \Nat$ and  $B_1, \ldots B_m \subset \RL^n$ be compact sets containing the origin. Then, for any compact  $B_0 \subset \RL^n$, is it true that
\ben
\left|\Delta_{B_1}\ldots \Delta_{B_m} (B_0)\right|\geq\sum_{s\subset[m]}(-1)^{m-|s|}\left|B_0+\sum_{i\in s} B_i\right| ?
\een
\end{ques}
The inequality (\ref{difference}) gives a positive answer to the above question in the case $m=2$.  We also observe that if $B_0, B_1, \dots, B_m$ are convex, then the right hand side is non-negative thanks to Theorem~\ref{thm:supmodvol}. We note that  it was observed in \cite{FMMZ18}, by considering $A=\{0,1\}$ and $B=C=[0,1]$ in $\RL^1$,
that the volume of Minkowski sums cannot be supermodular (even in dimension 1) if the convexity assumption on the set $A$ is removed.
Nonetheless \cite{FMMZ18} observed that if $A, B, C\subset \R$ are compact, then
\ben
|A+B+C| +|\conv(A)|\ge |A+B| + |A+C|;
\een
it is unknown if this extends to higher dimension. In particular, we do not know if the following conjecture is true for $n\ge2$.

\begin{conj}\label{submodulnotconv} For any convex body $A$  and any compact sets $B$ and $C$ in $\R^n$,
$$|A+B+C| +|A|\ge |A+B| + |A+C|. $$
\end{conj}
We can confirm Conjecture \ref{submodulnotconv} under the assumption that $B$ is a zonoid. 

\begin{thm}
 Assume $A$ is a convex compact set, $B$ is a zonoid and $C$ is any compact set in $\RL^n$. Then
$$|A+B+C| +|A|\ge |A+B| + |A+C|. $$
\end{thm}
\begin{proof}
By approximation, we may assume that $B$ is a zonotope.
Using the definition of mixed volumes (\ref{eq:mvf})
 and (\ref{eq:proj1}) we get that for any convex compact set $M$ in $\RL^n$
 $$
|M+[0,tu]|- |M|=t|P_{u^\perp} M|_{n-1}, \mbox{ for all } t>0, u\in  S^{n-1}.
 $$
 The above formula can be also proved using a geometric approach and thus studied in the case of not necessary convex $M$. Indeed, consider a compact set $M$ in $\RL^n,$ $t>0$ and $u\in S^{n-1},$ let $\partial_u M$ be the set of all $x \in \partial M$ such that $x \cdot u \ge y \cdot u,$ for all $y\in M$ for which $P_{u^\perp} y =P_{u^\perp}x.$ Note that
 $$
 (\partial_u M + (0,tu]) \cap M =\emptyset, \mbox{ but } M \cup (\partial_u M+ (0, tu]) \subseteq M+[0,tu].
 $$
 Thus 
 $$
|M+[0,tu]|  \ge |M|+  t|P_{u^\perp} M|_{n-1}.
 $$
 Now, we are ready to prove the theorem with $B=[0,tu]$
$$
      |A+C+[0,tu]| -|A+C| \ge t|P_{u^\perp} (A+C)|_{n-1}
      \ge t |P_{u^\perp} A|_{n-1} 
      = |A+[0,tu]|  -|A|.
$$
Thus, we proved that, for any $u \in \R^n$,
\be\label{in:vec}
|A+C+[0,u]| - |A+[0,u]|   \ge |A+C|   -|A|.
\ee
Now, we can prove the theorem for the case of a zonotope. Indeed, let $Z_k=\sum_{i=1}^{k} [0, u_i]$ be a zonotope. Apply inequality  (\ref{in:vec}) to the convex body $A+Z_{k-1}$ and the vector $u=u_{k}$ to get 
$$
|A+C+Z_{k}| -|A+Z_{k}| \ge  |A+C+Z_{k-1}| -|A+Z_{k-1}|.
$$
Iterate the above inequality to prove the theorem for the case of $B$ being a zonotope. The theorem now follows from continuity of the volume and the fact that every zonoid is a limit of zonotopes. 
\end{proof}

\section{Pl\"unnecke-Ruzsa inequalities for convex bodies}
\label{sec:Pl\"unnecke}

\subsection{Existing Pl\"unnecke-Ruzsa inequality for convex bodies}
\label{ss:plun}

Bobkov and Madiman \cite{BM12:jfa} developed a technique for going from entropy to volume estimates,
by using certain reverse H\"older inequalities that hold for convex measures. 
Specifically,  \cite[Proposition 3.4]{BM12:jfa} shows that if $X_i$ are independent random variables
with $X_i$ uniformly distributed on a convex body $K_i \subset \RL^n$ for each $i=1,\ldots, m$,
then 
$h(X_1+\ldots+X_m)\geq \log |K_1+\ldots+K_m| -n\log m$, where the entropy of a random variable $X$ with density $f$ on $\RL^n$ is defined by 
\begin{equation}\label{eq:entropy}
h(X)=-\int f(x)\log f(x) dx.
\end{equation}
This is a reverse H\"older inequality in the sense that $h(X_1+\ldots+X_m)\leq \log |K_1+\ldots+K_m|$ may be seen
by applying H\"older's inequality and then taking a limit.  More general sharp inequalities relating 
R\'enyi entropies of various orders for measures having convexity properties are described in \cite{FLM20} (see also \cite{BM11:it, FMW16, BFLM17}). Applied to the submodularity of entropy of sums discovered in \cite{Mad08:itw}, they use this technique to demonstrate the following inequality.

%

\begin{thm}\label{thm:frac-Pl\"unnecke}
Let $\collS_k$ denote the collection of all subsets of $[m]=\{1,\ldots, m\}$ that are of cardinality $k$.
Let $A$ and $B_{1},\ldots, B_{m}$ be convex bodies in $\RL^{n}$, and suppose
$$
\bigg|A+\sum_{i\in \setS} B_i \bigg|^{\nth} \leq c_{\setS} |A|^{\nth}
$$
for each $\setS\in\collS_k$, with given numbers $c_{\setS}$.
Then
$$
\bigg|A+\sum_{i=1}^m B_i \bigg|^{\nth}
\leq (1+m) \bigg[\prod_{\setS\in\collS_k} c_{\setS}\bigg]^{\frac{1}{\binom{m-1}{k-1}}} |A|^{\nth} .
$$
\end{thm}

In particular, by choosing $k=1$, one already obtains an interesting inequality for volumes of Minkowski sums.

\begin{cor}\label{cor:Pl\"unnecke}
Let $A$ and $B_{1},\ldots, B_{m}$ be convex bodies in $\RL^{n}$. Then
$$
|A|^{m-1}\bigg|A+\sum_{i=1}^m B_i \bigg|
\leq (1+m)^n \prod_{i=1}^m |A+B_i|  .
$$
\end{cor}


Thus, one may think of Corollary~\ref{cor:Pl\"unnecke} as providing yet another continuous analogues of the Pl\"unnecke-Ruzsa inequalities
in the context of volumes of convex bodies in Euclidean spaces (compare with (\ref{eq:ruzvol})), where going from the discrete to the continuous
incurs the extra factor of $(1+m)$, but one does not need to bother with taking subsets of the set $A$.
In particular, with $m=2$, one gets ``log-submodularity of volume up to an additive term'' on convex bodies.

\begin{cor}\label{cor:Pl\"unnecke2}
Let $A$ and $B_{1}, B_{2}$ be convex bodies in $\RL^{n}$. Then
\be\label{3body-3n}
|A|\,|A+B_1+B_2 | 
\leq 3^n  |A+B_1| \,|A+B_2|.
\ee
\end{cor}
Unfortunately the dimension-dependent additive term is a hindrance that one would like to remove or improve, which is the purpose of the next section.

\begin{rem} We notice that in the case where $B_1=B_2=B$ the inequality holds with constant $1$:
$$
|A|\,|A+B+B |
\leq   |A+B|^2.
$$
by the Brunn-Minkowski inequality. In the next section, we shall see that it is no longer true for $B_1\neq B_2$. Moreover, as we will see in Lemma \ref{lm:example},
the above inequality is not true with any absolute constant if we only assume that the sets $A$ and $B_1$ are compact, which  exposes an essential difference of this inequality with (\ref{eq:ruzvol}).
\end{rem}

%
%
%


\subsection{Improved upper bounds in general dimension}
\label{ss:gen-ub}

In this section, we will present an improvements in the constant $3^n$ in the three body inequality from Corollary~\ref{cor:Pl\"unnecke2}. We define the constant $c_n$ by \eqref{cn-def}, or equivalently as the infimum of the constants $c>0$ such that, for every convex compact sets $A,B,C$ in $\RL^n$,
\ben
|A|\,|A+B+C| \le c |A+B|\,|A+C|.
\een
We recall that $\varphi=(1+\sqrt{5})/2$ denotes the golden ratio.

\begin{thm}\label{thm:ub} Let $n\ge2$. Then, one has $1=c_2\le c_n\le \varphi^n$, i.e., for every convex compact sets  $A, B, C \subset \RL^n$, 
$$
|A|\,|A+B+C| \le \varphi^{n} |A+B|\,|A+C|.
$$
\end{thm}
\begin{proof}
Observe that, taking $B=C=\{0\}$ we get $c_n\ge1$.
We  apply (\ref{eq:mvf}) to get 
\begin{align*}
|A|\,|A+B+C|=&\sum_{k+j+m=n} {n \choose k,j,m} |A| V(A[k], B[j],C[m])\\=&\sum_{0\le j+m\le n} {n \choose j,m,n-j-m} |A| V(A[n-j-m], B[j],C[m]),
\end{align*}
$$
|A+B|\,|A+C|=\sum_{j=0}^n \sum_{m=0}^n {n \choose j} {n \choose m}   V(A[n-j], B[j]) V(A[n-m], C[m]).
$$
 The comparison of the above sums term by term shows that $c_n\le d_n$ where $d_n$ satisfies
\begin{equation} 
 |A| V(A[n-j-m], B[j],C[m]) \le d_n\frac{ {n \choose j} {n \choose m}}{ {n \choose j,m,n-j-m}}   V(A[n-j], B[j]) V(A[n-m], C[m]).
 \end{equation}
  Rewriting the above in a  more symmetric way we get, for $m, j \ge 0$ and $m+j\le n$:
\begin{eqnarray}\label{eq:conj-mixed}
\frac{|A|}{n!}\frac{V(A[n-j-m], B[j], C[m])}{(n-j-m)!}\le d_n\frac{V(A[n-j],B[j])}{(n-j)!}\frac{V(A[n-m],C[m])}{(n-m)!}.
\end{eqnarray}
Notice that for $m=0$ or $j=0$, (\ref{eq:conj-mixed}) trivially holds for any  $d_n\ge 1$. 
Using inequality (\ref{eq:jxiao})  we get that $d_n$ will satisfy inequality (\ref{eq:conj-mixed}) as long as 
\begin{equation}\label{est_b}
 \min\left\{ {n-j  \choose m}, {n-m \choose j} \right\} \le d_n. 
\end{equation}
Note that the above is true with constant $d_n=1$ if $m+j=n$. 
We also note that, if $m=j=1$, then the required inequality (\ref{eq:conj-mixed}) becomes
\begin{eqnarray}
|A|V(A[n-2], B, C)\le d_n\frac{n}{n-1} V(A[n-1],B) V(A[n-1],C).
\end{eqnarray}
Using (\ref{alexloc}), we see that in this case it is enough to select   $d_n \ge \frac{2(n-1)}{n}$. In particular, we get that $c_2=d_2=1$. For the more general case, we can provide a bound for  $d_n^{1/n}$ using Stirling's approximation formula. Indeed,
$$
{p \choose q}\le \frac{p^p}{(p-q)^{p-q} q^q} e^{\frac{1}{12p}-\frac{1}{12(p-q)+1}-\frac{1}{12q+1}} \sqrt{\frac{p}{2\pi (p-q)q}}
$$
$$
{p \choose q}\le \frac{p^p}{(p-q)^{p-q} q^q} e^{\frac{1}{12p}-\frac{12p+2}{(12(p-q)+1)(12q+1)}} \sqrt{\frac{p}{2\pi (p-1)}} \le \frac{p^p}{(p-q)^{p-q} q^q}.
$$
Next, let $j=yn$ and $m=xn$, where $x,y \ge 0$, $x+y \le 1,$ then it is sufficient for $d_n$ to satisfy
$$
\max_{\substack{x, y \ge 0\\ x+y\le 1}}\min\left\{\frac{(1-y)^{1-y}}{(1-x-y)^{1-x-y}x^x}, \frac{(1-x)^{1-x}}{(1-x-y)^{1-x-y}y^y} \right\}\le d_n^{1/n}.
$$
Without loss of generality we may assume that $|x-1/2| \le |y-1/2|$ and thus $(1-x)^{1-x}x^x \le (1-y)^{(1-y)}y^y$. Our next goal is to provide an upper estimate for 
$$
\max \frac{(1-x)^{1-x}}{(1-x-y)^{1-x-y}y^y} ,
$$
where the maximum is taken over a set 
\begin{align*}
    \Omega&=\{(x, y)\in\R_+^2:  x+y\le 1, |1/2-x| \le |1/2-y|\}\\
    &=\{(x,1-x); 0\le x\le 1/2\}\cup\{(x,y)\in\R_+^2;  y \le \min(x,1-x)\}.
\end{align*}
We note that the function $y\mapsto(1-x-y)^{1-x-y}y^y$ is decreasing for $y \in [0, (1-x)/2]$ and increasing on $[(1-x)/2, (1-x)]$. So we may consider two cases, comparing $x$ and $(1-x)/2$. Next,
$$
\max_{\Omega \cap \{x\in [0,1/3]\}} \frac{(1-x)^{1-x}}{(1-x-y)^{1-x-y}y^y} = \max_{[0,1/3]} \frac{(1-x)^{1-x}}{(1-2x)^{1-2x}x^x} = \frac{1+\sqrt{5}}{2}.
$$
The last equality follows from the fact that the maximum is achieved when
$
\frac{(1-2x)^2}{(1-x)x}=1, 
$
i.e. at $x=(5-\sqrt{5})/10$. Finally
$$
\max_{\Omega \cap \{x\in [1/3,1]\}} \frac{(1-x)^{1-x}}{(1-x-y)^{1-x-y}y^y} \le \max_{x\in [1/3,1]} \frac{(1-x)^{1-x}}{((1-x)/2)^{1-x}} = 2^{2/3} < \frac{1+\sqrt{5}}{2}.
$$
\end{proof}


 The next proposition  gives a different proof of (\ref{eq:ruzvol1})  in the special case of convex sets and, we hope, gives yet another example of how the methods of mixed volumes as well as the B\'ezout type inequality (\ref{eq:jxiao}) can be applied in this context.
\begin{prop}\label{lem:weak-3body}
Let  $A, B_1, \dots, B_m$ be convex bodies in $\RL^n$, then
\begin{equation}\label{eq:threesimple}
|A|\,|B_1+\dots +B_m| \le \prod\limits_{i=1}^m|A+B_i|,
\end{equation}
with equality if and only if $|A|=0$.
\end{prop}
\begin{proof}
By induction, the general case follows immediately from the case $m=2$, so we assume $m=2$ and denote $B_1=B$ and $B_2=C$. 
The inequality follows from the proof of Theorem \ref{thm:ub} and the observation that decomposing the left and right hand sides of the (\ref{eq:threesimple}) we need to show that
$$
\sum_{m=0}^{n} {n \choose m} |A| V(B[n-m], C[m])  
\le \sum_{j=0}^n \sum_{m=0}^n {n \choose j} {n \choose m}   V(A[n-j], B[j]) V(A[n-m], C[m]).
$$
It turns out that it is enough to consider only the terms with $m+j=n$ on the right hand side, i.e. to show that 
$$
 \sum_{m=0}^{n} {n \choose m}  |A| V(B[n-m], C[m])  
\le \sum_{m=0}^{n}  {n \choose n-m} {n \choose m}  V(A[m], B[n-m]) V (A[n-m], C[m]),
$$
which is true term by term by using  (\ref{eq:jxiao}) with $m+j=n$. Now assume that there is equality. This implies that the term $j=m=0$ in the above double sum must vanish, i.e
$|A|=0$. 
\end{proof}

\subsection{Improved constants in dimensions 3 and 4}
\label{ss:ub34}

Theorem \ref{thm:ub} gives an optimal bound of  $1$ for the three body inequality in dimension $2$. Next, we will show how we can get better bounds for $c_n$ in dimension $3$ and $4$. 

\begin{thm}\label{th:r3} Let $A, B, C$ be convex compact sets in $\RL^3$ then
$$
|A|\,|A+B+C| \le \frac{4}{3} |A+B|\,|A+C|
$$
and the constant is best possible: $c_3=\frac{4}{3}$. Moreover, if $A$ is a simplex, then
$$
|A|\,|A+B+C| \le |A+B|\,|A+C|.
$$
\end{thm}
\begin{proof} We follow the same strategy as in the proof of Theorem \ref{thm:ub} and arrive to the inequality (\ref{eq:conj-mixed}) with  $m, j \ge 0$ and $m+j\le 3$:
$$
\frac{|A|}{3!}\frac{V(A[3-j-m], B[j], C[m])}{(3-j-m)!}\le d_3\frac{V(A[3-j],B[j])}{(3-j)!}\frac{V(A[3-m],C[m])}{(3-m)!}.
$$
Again, the inequality is trivially true for $m=0$ or $j=0$  with any constant $d_3\ge1$. Thus,  we are left with the two following inequalities
\begin{equation}\label{eqXiao}
|A|V(C, B[2])\le3d_3V(B,A[2])V(A,C[2]) 
\end{equation}
\begin{equation}\label{eq:bez}
|A|V(A,B,C) \le \frac{3}{2} d_3 V(B,A[2])V(C, A[2])
\end{equation}
We note that the inequality (\ref{eqXiao}) with $d_3\ge1$ follows from (\ref{eq:jxiao}).
Next we note that (\ref{eq:bez}) is true with $d_3 \ge 1$ when $A$ is a simplex (see \cite{SZ16}). The general case of  (\ref{eq:bez}) follows from  (\ref{alexloc}) with $c_3=4/3$. The proof that this bound is optimal is made in section \ref{ss:PR-LB}, where we, in particular, establish that $c_n\ge 2-\frac{2}{n}$. 
\end{proof}

\begin{thm} \label{th:r4} Let $A, B, C$ be convex compact sets in $\RL^4$, then
$$
|A|\,|A+B+C| \le 2 |A+B|\,|A+C|.
$$
Thus $c_4\le2$. Moreover, if $A$ is a simplex, then 
$$
|A|\,|A+B+C| \le |A+B|\,|A+C|.
$$
\end{thm}

\begin{proof}
We will  check inequality (\ref{eq:conj-mixed})    for $n=4$, $0\le m\le j\le4$ and $m+j\le 4$:
$$
\frac{|A|}{4!}\frac{V(A[4-j-m], B[j], C[j])}{(4-j-m)!}\le c_4\frac{V(A[4-j],B[j])}{(4-j)!}\frac{V(A[4-m],C[m])}{(4-j)!}.
$$
The inequality is trivially true for $m= 0$ or $j= 0$ with a constant $d_4 \ge 1$. Taking in account that the inequality is symmetric with respect to $m$ and $j$ and to $B$ and $C$ we get that it is enough to obtain cases $(j,m)=\{(1,1); (1,2); (1,3); (2,2)\}$.  For $(j,m)=(1,1)$, we need to obtain 
$$
|A| V(A[2],B,C) \leq \frac{4}{3} d_4 V(A[3],B) V(A[3],C).
$$
If $A$ is a simplex then the above is true with $\frac{4}{3} d_4=1$ (see \cite{SZ16}) and the general case, follows from  (\ref{alexloc}) with $\frac{4}{3} d_4 \ge 2$, that is $d_4\ge \frac{3}{2}$.  For $(j,m)=(1,2)$, we need to obtain 
$$
|A| V(A,B,C,C)\leq 2d_4 V(A[3],B) V(A[2],C[2]).
$$
We again observe that if $A$ is a simplex then the above is true with $2d_4=1$. To resolve the general case  we apply (\ref{eq:jxiao}) with $n=4$ and $(j,m)=(1,2)$, we get
 $2d_4 \ge 4$ will satisfy the requirement. 
 When $(j,m)=1,3$ we need to show that 
$$
|A|V(B, C[3])\le 4 d_4 V(A[3], B) V(A,C[3]),
$$
which,  from   (\ref{eq:jxiao}), is  true for  $4d_4\ge 4$ for all convex, compact sets.
Finally, when $(j,m)=(2,2)$ we need to obtain
$$
|A|V(A, B[2], C[2])\le 6 d_4 V(A[2],B[2])V(A[2],C[2]),
$$
which is from (\ref{eq:jxiao}) true for $6 d_4 \ge 6$ for all convex compact sets.
\end{proof}

\begin{rem} 
We conjecture that actually $c_4=3/2$.
\end{rem}

\begin{rem} We conjecture that, for $1\leq j\leq n$, 
$$
|A| V(L_1, \ldots, L_j, A[n-j]) \leq j V(L_j, A[n-1]) V(L_1, \ldots, L_{j-1}, A[n-j+1]).
$$
 The inequality (which is an improvement of a special case of (\ref{eq:jxiao}) would help to   obtain the best constant in  $\RL^4$ and corresponds to the  $(j,m)=(1,2)$ in the proof of Theorem \ref{th:r4}.

\end{rem}

\subsection{Lower bounds in general dimension}
\label{ss:PR-LB}

In this section, we provide a lower bound for the Pl\"unnecke-Ruzsa inequality for convex bodies.  A weaker lower bound was also independently obtained by Nayar and Tkocz \cite{NT17}.
We first observe that the best constant $c_n$ in the Pl\"unnecke-Ruzsa inequality
\begin{equation}\label{PR2}
|A|\,|A+B+C| \le c_n |A+B|\,|A+C|,
\end{equation}
satisfies $c_{n+m} \ge c_n c_m$. Indeed, this follows immediately by considering critical examples of $A_1, B_1, C_1$ in $\RL^n$ and $A_2, B_2, C_2$ in $\RL^m$ together with their direct products
$A_1 \times A_2, B_1\times B_2, C_1\times C_2$ in $\RL^{n+m}$.
Next we notice that if (\ref{PR2}) is true in a class of convex bodies closed by linear transformations, then
\begin{equation}\label{projcontr}
|P_{E \cap H}K| |K| \le  c_n |P_E K|\,|P_H K|,
\end{equation}
for any $K$ in this class and any subspaces $E,H$ of $\RL^n,$ such that $\dim E =i,$ $\dim H=j,$ $i+j \ge n+1$ and $E^\perp \subset H$. To see this consider $B=U$, with $\dim U = n-i,$ $|U|=1$ and $C=V$, with $\dim V = n-j,$ $|V|=1$ and  $U, V$  belong to orthogonal subspaces of $\R^n$. Let $A=tK$, where $t > 0$ and set  $k=n-(n-i)-(n-j)=i+j-n$. Then (\ref{PR2}) yield  together with (\ref{eq:mvf}) and (\ref{eq:proj}) 
\begin{align*}
t^n|K|\big(\sum_{m=k}^n &{{n}\choose{m}}V(K[m], (U+V)[n-m])t^m \big) \\ &\le  c_n\big(\sum_{m=i}^n {{n}\choose{m}}V(K[m], U[n-m])t^m \big)\big(\sum_{m=j}^n {{n}\choose{m}}V(K[m], V[n-m])t^m \big).
\end{align*}
Dividing the above inequality by $t^{n+k}$ and taking  $t= 0$ we get
$$
 |K|{{n}\choose{k}}V(K[k], (U+V)[n-k]) \le  c_n 
{{n}\choose{i}}V(K[i], U[n-i]) {{n}\choose{j}}V(K[j], V[n-j]).
$$
Finally, using (\ref{eq:proj}), we get  (\ref{projcontr}).
It was proved in \cite{GHP02} that 
\begin{equation}\label{eqGPH}
|P_{\{u,v\}^\perp}K| |K| \le  \frac{2(n-1)}{n} |P_{u^\perp} K|\,|P_{v^\perp}K|,
\end{equation}
for any convex body $K \subset \RL^n$ and a pair of orthogonal vectors $u,v \in S^{n-1}$. It was also shown in \cite{GHP02} that the constant $2(n-1)/n$ is optimal. Thus $c_n \ge 2 -\frac{2}{n}$ and this estimate gives a sharp constant in $\R^3$:  $c_3 = 4/3.$ In the case when $n=4$, we get $c_4 \ge 3/2$. 
The inequalities analogous to (\ref{eqGPH}) and (\ref{projcontr}) were studied in many other works, including \cite{FGM03, SZ16, AFO14, AAGHV17}. In particular, it was proved in \cite{AAGHV17} that   (\ref{projcontr}) is sharp with
$$
c_n \ge c_n(i,j,k)=\frac{{{i}\choose{k}}{{j}\choose{k}}}{{{n}\choose{k}}}.
$$
 Thus to find a lower bound on  $c_n$  one may maximize over $c_n(i,j,k)$ with restriction that  $i+j \ge n+1$ and $k=i+j-n$.  One may use Stirling's approximation, with $i=j=2n /3$  and $k=n/3$ (when $n$ is a multiple of 3, with minor modifications if not) to obtain the following theorem.
 
\begin{thm}\label{thm:lb}
For sufficiently large $n$, we have that
$
c_n
\geq \frac{2}{\sqrt{\pi n}} \left(\frac{4}{3}\right)^n
= \left(\frac{4}{3} +o(1)\right)^n.
$
\end{thm}


\subsection{Improved upper bound for subclasses of convex bodies}
\label{ss:improved}

The goal of this section is to prove the following  theorem 
\begin{thm}\label{thm:zonoid-ellipsoid}
Let $n\ge1$ and $K$ be a convex body in $\RL^n$. Let $B$ be an ellipsoid and $Z$ be a zonoid in $\RL^n$. Then 
$$
|B|\,|B+K+Z|\le |B+K|\,|B+Z|.
$$
\end{thm}

Theorem~\ref{thm:zonoid-ellipsoid}  motivates us to pose the following conjecture.

\begin{conj}
Let $n\ge1$ and $A, B, C$ be zonoids in $\RL^n$. Then 
$$
|A|\,|A+B+C|\le |A+B|\,|A+C|.
$$
\end{conj}

A detailed study of this conjecture is undertaken in the forthcoming paper \cite{FMMZ22}.



Before proving Theorem \ref{thm:zonoid-ellipsoid} we will prove a  theorem which would help us to verify Pl\"unnecke-Ruzsa inequalities for convex bodies for a fixed body $A$.

\begin{thm}\label{thm:zonoid-projratio}
Let $n\ge1$ and $A,B$ be a convex bodies in $\RL^n$ such that for every $1\le k\le n$ and any subspace $E$ of $\RL^n$ of dimension $k$ one has 
\[
\frac{|P_E(A+B)|_k}{|P_{E\cap u^\bot} (A+B)|_{k-1}}\geq \frac{|P_EA|_k}{|P_{E\cap u^\bot}A|_{k-1}}.
\]
Then for any zonoid $Z$ in $\RL^n$ one has
$$
|A|\,|A+B+Z|\le |A+B|\,|A+Z|.
$$
\end{thm}

\begin{proof}
Notice that it is enough to prove the inequality for $Z$ being a zonotope and use an approximation argument. In fact, we prove by induction on $k$ that for any $1\le k\le n$ and any subspace $E$ of $\RL^n$ of dimension $k$ and zonoid $Z$ in $E$  one has
\be\label{eq:induc:sum-proj}
|P_EA|_k|P_E(A+B)+Z|_k\le |P_E(A+B)|_k|P_EA+Z|_k.
\ee
This statement is true for $k=1$ so let us assume that it's true for $k-1$, for some $1\le k\le n$ and let's prove it for $k$. Let $E$ be a subspace of dimension $k$. To prove that inequality (\ref{eq:induc:sum-proj}) holds for any zonotope in $E$, we proceed by induction on the number of segments in $Z$.
Notice that the inequality holds as an equality for $Z=0$. Let us assume that inequality (\ref{eq:induc:sum-proj}) holds for some fixed zonotope $Z$ in $E$ and prove it for $Z+t[0,u]$ where $t>0$ and $u\in S^{n-1} \cap E$. 
Using (\ref{eq:ste}), we get 
\[
|P_E(A+B)+Z+t[0,u]|_{k}=|P_E(A+B)+Z|_k+t|P_{E\cap u^\bot}(A+B+Z)|_{k-1}
\]
and 
$$
|P_EA+Z+[0,tu]|_k=|P_EA+Z|_k+t|P_{E\cap u^\bot}(A+Z)|_{k-1}.
$$
Applying the induction hypothesis, it is enough to prove that 
\begin{equation}\label{inductiontwob}
 |P_EA|_k|P_{E\cap u^\bot}(A+B+Z)|_{k-1} \le |P_{E\cap u^\bot}(A+Z)|_{k-1} |P_E(A+B)|_k.
\end{equation}
But the inequality in the $k-1$-dimensional subspace $E\cap u^\perp$ for the zonotope $P_{u^\perp} Z$ gives 
$$
|P_{E\cap u^\bot}(A)|_{k-1}|P_{E\cap u^\bot}(A+B+Z)|_{k-1}  \le |P_{E\cap u^\bot}(A+B)|_{k-1} |P_{E\cap u^\bot}(A+Z)|_{k-1}.
$$
Multiplying this inequality by the assumption of the theorem: 
\[
 \frac{|P_EA|_k}{|P_{E\cap u^\bot}A|_{k-1}} \le \frac{|P_E(A+B)|_k}{|P_{E\cap u^\bot} (A+B)|_{k-1}},
\]
we get (\ref{inductiontwob}).
\end{proof}

Next we will prove that $B_2^n$ satisfies the conditions of theorem \ref{thm:zonoid-projratio}.

\begin{thm}\label{thm:projection-ball}
Let $n\ge1$ and $K$ be a compact set in $\RL^n$. Let $u\in S^{n-1}$. Then 
$$
\frac{|K+B_2^n|}{|P_{u^\bot}(K+B_2^n)|_{n-1}}\ge\frac{|B_2^n|}{|B_2^{n-1}|_{n-1}}.
$$
\end{thm}


\begin{proof}
We will use a trick from \cite{AFO14} to reduce the proof to the case of $K \subset u^\perp$. Indeed, let $S_u$ be the Steiner symmetrization with respect to $u\in S^{n-1}$ (see \cite{Sch14:book} and \cite{BZ88:book} remark 9.3.2). Then one has 
$$
S_u(K+B_2^n)\supset S_u(K)+S_u(B_2^n)=S_u(K)+B_2^n.
$$
Hence 
$$
|K+B_2^n|=|S_u(K+B_2^n)|\ge |S_u(K)+B_2^n|.
$$
Moreover, $P_{u^\bot}(S_u(K))=P_{u^\bot}K$ and $P_{u^\bot}K\subset K$, hence,
$$
|K+B_2^n|\ge|P_{u^\bot}K+B_2^n|.
$$
It follows that we are reduced to the case when $K\subset u^\bot$, {\em i.e.} $K=P_{u^\bot}K$. 
Without loss of generality, we may assume that $u=e_n$ and we write $B_2^n\cap u^\bot=B_2^{n-1}$.
In this case, we can describe the set $K+B_2^n$ by its slices by the hyperplanes orthogonal to $e_n$ denoted $H_t=\{x\in\RL^n; x_n=t\}$, $t\in\RL$.
We have 
$$
(K+B_2^n)\cap H_t=K+B_2^n\cap H_t=K+\sqrt{1-t^2}B_2^{n-1}+te_n.
$$ 
Using (\ref{eq:ste}), we get that for all $t\in[-1,1]$
$$
|(K+B_2^n)\cap H_t|_{n-1}=\left|K+\sqrt{1-t^2}B_2^{n-1}\right|_{n-1}.
$$
It follows from Proposition 2.1 \cite{FM14}  (see also \cite{Sta76}) that
$$
\left|K+\sqrt{1-t^2}B_2^{n-1}\right|_{n-1} \ge  \left|\sqrt{1-t^2} K+\sqrt{1-t^2}B_2^{n-1}\right|_{n-1}.
$$
Using Fubini's theorem, we get
$$
|K+B_2^n|= \int_{-1}^1|(K+B_2^n)\cap H_t|_{n-1}dt \ge \left| K+B_2^{n-1}\right|_{n-1}\int_{-1}^1(1-t^2)^\frac{n-1}{2}dt.
$$
We finish the proof by noticing that 
$$
\int_{-1}^1(1-t^2)^\frac{n-1}{2}dt=\frac{|B_2^n|}{|B_2^{n-1}|}.
$$
\end{proof}

\noindent{\it Proof of Theorem  \ref{thm:zonoid-ellipsoid}:}
Let $T$ be the affine transform such that $B=T(B_2^n)$. If $B$ lives in an hyperplane then $|B|=0$ and the inequality holds. If not, then $T$ is invertible and since the affine image of a zonoid is a zonoid by applying $T^{-1}$ we may assume that $B=B_2^n$. Now the theorem follows immediately from Theorems \ref{thm:projection-ball} and \ref{thm:zonoid-projratio}. {\begin{flushright} $\Box $\end{flushright}}

\begin{rem}
By applying a linear transform, it is possible to show that, more generally, for any compact set $K$ and for any ellipsoid $B$,
$$
\frac{|K+B|}{|P_{u^\bot}(K+B)|_{n-1}}\ge\frac{|B|}{|P_{u^\bot}B|_{n-1}}.
$$
Indeed, let $T\in GL_n$, such that $B=TB_2^n$ denoting $H=(T^*u)^\bot$, one has, that for any compact $K$,  $P_{u^\bot}(TK)=TP_{H,T^{-1}u}K$, where $P_{H,T^{-1}u}K$ is the linear projection of $K$ onto $H$ along $T^{-1}u$.
Then the remark follows by applying Theorem \ref{thm:projection-ball} to $TK$.
\end{rem}

\begin{cor} Let $n\ge 1$ and $K$ be a convex body in $\RL^n$. Let $B$ be an ellipsoid and $Z_1, \dots, Z_m,$ $m\ge 1$ be  zonoids in $\RL^n$. Then 
$$
|B|^{m}|B+K+Z_1+\dots Z_m|\le |B+K|\prod_{i=1}^m|B+Z_i|.
$$
\end{cor}
\begin{proof} The corollary follows immediately applying induction on $m \ge 1$ together with   Theorem \ref{thm:zonoid-ellipsoid}:
$$
|B|^{m}|B+K+Z_1+\dots Z_m|\le |B|^{m-1} |B+K+ Z_1+\dots Z_{m-1}|\,|B+Z_m|.
$$
\end{proof}

\begin{rem}
Theorem \ref{thm:ub} was inspired by the following inequality 
\begin{equation}\label{eq:HS}
V(B_2^n,Z)V(B_2^n,K)\geq \frac{n-1}{n} c_n |B_2^n|V(B_2^n[n-2],Z,K),
\end{equation}
where $c_n =|B_2^{n-1}|^2/(|B_2^{n-2}|\,|B_2^n|)>1$. This inequality was proved by Artstein-Avidan, Florentin, and Ostrover \cite{AFO14} when $Z$ is a zonoid and $K$ is an arbitrary convex body, as a generalisation of of a result of Hug and Schneider \cite{HS11:1} who proved it for $K$ and $Z$ zonoids.
It is an interesting question if one can prove directly Theorem~ \ref{thm:zonoid-ellipsoid} by using the decomposition into mixed volumes as it was done in the proof of Theorem \ref{thm:ub} and applying inequality \eqref{eq:HS}.  Inequality (\ref{eq:HS}) is a sharp improvement of (\ref{eq:jxiao}) in the case when $A=B_2^n$ and $m=j=1,$ and one of the bodies is a zonoid.  Unfortunately, there seems not to be a direct way to apply (\ref{eq:HS}) to prove Theorem  \ref{thm:zonoid-ellipsoid} due to the lack of a sharp analog of this inequality for $V(B_2^n[n-2],Z[m],K[j])$ when $m,j >1$.
\end{rem}

\subsection{The case of compact sets}
\label{ss:compact}

Let us note that inequality, 
\be\label{eq:cnonconv}
|A|\,|A+B+C|\le |A+B|\,|A+C|.
\ee
is valid, when $A,B$ are intervals and $C$ is any compact set in $\RL$. Indeed, by approximation we may assume that $C$ is a finite union of closed intervals,  $A=[0,a]$ and $B=[0,b],$  for some $a, b \ge 0$. Then we may assume that
$A+C=\cup_{i=1}^m [\alpha_i, \beta_i+a]$ where intervals $[\alpha_i, \beta_i+a]$ are mutually disjoint. Then
$$
|A+C| =  ma+\sum\limits_{i=1}^m (\beta_i-\alpha_i),
$$
$$
|A+B+C| \le \sum \limits_{i=1}^m (\beta_i-\alpha_i +a+b)=m(a+b)+\sum\limits_{i=1}^m (\beta_i-\alpha_i)
$$
and (\ref{eq:cnonconv}) follows from
$$
a (m(a+b)+\sum\limits_{i=1}^m (\beta_i-\alpha_i))\le (ma+\sum\limits_{i=1}^m (\beta_i-\alpha_i))(a+b).
$$
We also note that, as we discussed before,  inequality (\ref{eq:ruzvol}) as well as inequality (\ref{eq:threesimple}) is valid without additional  convexity assumptions (as well as Theorem \ref{thm:projection-ball} from above). Still,  we will show that there is a  sharp difference to those inequalities the convexity assumption in Theorem \ref{thm:ub} can not be removed. The  construction is inspired by the proof of Theorem  7.1 from \cite{Ruz96:1}:
\begin{lem}\label{lm:example}
Fix $n\ge 1$, then  for any $\beta >0$ there exist two compact sets $A, B\subset \R^n$ such that
$$
|A|\,|A+B+B| > \beta |A+B|^2.
$$
\end{lem}
    \begin{proof} It is enough to prove the theorem for the case $n=1$, indeed, for any other dimension $n>1$, one can consider  $A \times [-1, 1]^{n-1}$, $B\times [-1, 1]^{n-1}$ where $A,B$ are the example constructed  in $\R$.

To construct the sets $A,B \subset \R$, we fix $m, l \in \N$ large enough and define first two discrete sets $A'$ and $B'$ in $\R$ to establish the analogue result for cardinality instead of volume. Let
$$
A'=\{x+y\sqrt{2}: x,y \in \{0,1,\dots, m-1\}\} \cup \{z\sqrt 3, z \in \{1,\dots, m\}\},
$$
thus $\#A'=m(m+1).$ We also define
$$
B'=\{x: x \in \{0,1,\dots, l\}\} \cup \{y\sqrt 2, y \in \{1,\dots, l\}).
$$
Thus, one has
$$
B'+B'= \{x+y\sqrt{2}: x,y \in \{0,1,\dots, l\}\} \cup \{x, x\in \{l+1,\dots, 2l\}\}\cup \{y\sqrt{2}, y\in \{l+1,\dots, 2l\}\}.
$$
It follows that  $\#(B'+B')=l^2+4l+1$. It may help to imagine $A'$ and $B'$ as $3$-dimensional subsets which are linear combinations of vectors $e_1, \sqrt{2}e_2$ and $\sqrt{3}e_3$.  Then it is easy to see    that $A'+B'$ consists of an $m\times m$ square of integer points united with  two $m\times l$ rectangles of integer points in the $\{e_1, e_2\}$ plane and two additional rectangles: one  of size $m\times (l+1)$ one in the $\{e_1, e_3\}$ plane  and another  of size $m\times (l+1)$ in the $\{e_2, e_3\}$ plane, where the last two rectangles overlap by $m$ points. Thus
$$
\#(A'+B')=m^2 +2ml +2 (m(l+1)-m)=m(m+4l+1).
$$
Finally, we note that $A'+B'+B'$ contains the set 
$$
\{z\sqrt{3}; z\in \{1,\dots, m\}\} + B'+B',
$$
thus $\#(A'+B'+B')\ge l^2 m$. Now consider any $\beta'>0$. Our goal is to select $m,l \in \N$ such that
\begin{equation}\label{eq:rusd}
\#(A')\#(A'+B'+B') > \beta' (\#(A'+B'))^2.
\end{equation}
For this it is enough to pick $m,l \in \N$ such that
$$
m(m+1) l^2 m \ge \beta' m^2(m+4l+1)^2
$$
or
$$
\sqrt{m+1} l \ge \sqrt{\beta'} (m+4l+1),
$$
which is true as long as $m+1> 16\beta'$ and $l$ is large enough. 

Now we are ready to construct our continuous example in $\R$, for volume. For the fixed $\beta>0$ consider $\beta'=\frac{4}{3}\beta$ and   sets $A'$ and $B'$ satisfying (\ref{eq:rusd}). Define $A=A'+[-\varepsilon, \varepsilon]$ and $B=B'+[-\varepsilon, \varepsilon]$ where  $\epsilon>0$ small enough such that 
$$
|A|=2\varepsilon \#(A'); |A+B|=4\varepsilon \#(A'+B') \mbox{ and } |A+B+B|=6\varepsilon \#(A'+B'+B'),
$$
which, together with (\ref{eq:rusd}) gives  the required inequality.
\end{proof}

It turns out (see for example \cite{FLZ19}) that some sumsets estimates can still be proved if the convexity assumption is relaxed by an assumption that the body is star-shaped. The next lemma shows that it is still not the case for  Theorem \ref{thm:ub}.
\begin{lem}
Fix $n\ge 3$, then  for any $\beta >0$ there exist a compact star-shaped symmetric body $A\subset \R^n$ such that
$$
|A|\,|A+A+A|\ge\beta|A+A|^2.
$$
\end{lem}
\begin{proof}
Let $n=3$, consider a cube $Q=[-1,1]^3$ and a set
$
X=m([-e_1,e_1] \cup [-e_2,e_2] \cup [-e_3,e_3]),
$
i.e. the union of $3$ orthogonal segments of length $2m$. Then
$$
|A+A| \le |X+X|+|Q+Q|+|X+Q| \le 0+4^3+3*(2m+2)*2
$$
but 
$
|A+A+A| \ge |X+X+X| =8m^3.
$
We note that in dimension $n>3$ one can consider
the direct sum of the above three dimensional example with  $[-1,1]^{n-3}$.
\end{proof}

\section{On Ruzsa's triangle inequality}
\label{sec:diff}

In additive combinatorics, the Ruzsa distance is defined by 
$
d(A,B)=\log \frac{\#(A-B)}{\sqrt{\#(A) \#(B)}},
$
where $A$ and $B$ are subsets of an abelian group. We refer to \cite{TV06:book} for more information and properties of this object, which is useful even though it is {\it not} a metric (since typically $d(A,A)>0$). The Ruzsa distance satisfies the triangle inequality which is equivalent to
$\#(C-B)\cdot \# A \leq \#(A-C) \#(B-A).$


An analogue of Ruzsa's triangle inequality holds for compact sets in $\R^n$.

\begin{thm}(see, e.g., \cite[Lemma 3.12]{TV09:1}) 
For any compact sets $A,B,C \subset \R^n$,
\begin{equation}\label{triangleI}
|A|\,|B-C| \le |A-C|\,|A-B|.
\end{equation}
\end{thm}

This inequality has a short proof that we provide here for the sake of completeness. Indeed
\[
|B-A| |A-C|=\int_{\R^n} 1_{B-A}*1_{A-C}(x)dx\ge \int_{B-C} 1_{B-A}*1_{A-C}(x)dx,
\]
where $1_M$ is a characteristic function of a set $M \subset \R^n$ and $f*g$ is the convolution of functions $f,g: \R^n \to \R$. Now let $x\in B-C$, there is $b\in B$ and $c\in C$ such that $x=b-c$. Thus, changing variable, one has 
\begin{align*}1_{B-A}*1_{A-C}(b-c)&= \int_{\R^n}1_{B-A}(z)1_{A-C}(b-c-z)dz=\int_{\R^n}1_{B-A}(b-y)1_{A-C}(y-c)dy\\
&\ge\int_{A}1_{B-A}(b-y)1_{A-C}(y-c)dy =|A|.
\end{align*}





In view of Ruzsa's triangle inequality, it is natural to try to generalize Theorem \ref{thm:ub} to the case of the difference of convex bodies. We recall that $\varphi=(1+\sqrt{5})/2$ denotes the golden ratio and for $n \ge 2$ the constant $c_n$ was defined in Theorem \ref{thm:ub} satisfying  $1=c_2\le c_n\le \varphi^n$.

\begin{thm}\label{thm:ruzsa}
Let $A, B, C$ be convex bodies in $\RL^n$. Then
\begin{equation}\label{eq:ruzsa}
|A|\,|A+B+C| \le \frac{1}{2^n}{ 2n \choose n}c_n 
\min\{ |A-B|\,|A+C| , |A-B|\,|A-C|.\}
\end{equation}
\end{thm}
\begin{proof} 
We first recall Litvak's observation (see \cite[pp. 534]{Sch14:book}) 
that 
\begin{equation}\label{eq:lit}
|A+B| \le \frac{1}{2^n}{2n \choose n} |A-B|,
\end{equation}
for any convex bodies $A, B$  in $\RL^n$, with equality for $A=B$ being simplices. Litvak obtained this by simply combining the Rogers-Shephard inequality (applied to  $A+B$) and the Brunn-Minkowski inequality. 

Now we can write
$$
|A|\,|A+B+C| \le c_n|A+B|\,|A+C|
\le \frac{1}{2^n}{ 2n \choose n}c_n|A-B|\,|A+C|,
$$
where the first inequality follows from Theorem \ref{thm:ub}, and the
second from Litvak's observation. This gives the first half of the inequality (\ref{eq:ruzsa}) (i.e., the inequality with the first term inside the minimum).

Next, we notice that  changing $B$ to $-B$ and $C$ to $-C$, Theorem \ref{thm:ub} becomes equivalent to 
$$
|A|\,|A-(B+C)| \le c_n|A-B|\,|A-C|,
$$
and we finish the proof of the inequality (\ref{eq:ruzsa}) (i.e., we get the inequality with the second term inside the minimum) by applying  Litvak's observation on the left hand side.
\end{proof}

\begin{rem} We note that the constant in (\ref{eq:ruzsa}) 
is sharp in the case $n=2$, no matter which term inside the minimum we consider. Indeed, if $C=\{0\}$, then (\ref{eq:ruzsa}) becomes
$
|A+B| \le \frac{3}{2}
|A-B|,
$
which is sharp for triangles $A,B$ such that  $A=-B$.
\end{rem}


Next we would like to discuss an improvement of  Theorem \ref{thm:ruzsa} via an improvement of Litvak's inequality (\ref{eq:lit})  in $\R^2$.

\begin{thm}\label{lmcool2} Let $A, C$ be two convex sets in $\R^2,$ then
\begin{equation}\label{eqprs3}
 |A-C|\le |A+C|+2\sqrt{|A|\,|C|},
\end{equation}
moreover, the equality in the above inequality is only possible in the following cases
\begin{itemize}
    \item One of the sets $A$ or $C$ is a singleton or a segment and the other one is any convex body.
    \item $A$ is a triangle and $C=t A+b,$ for some $t>0$ and $b\in \R^2.$
\end{itemize}
\end{thm}
\begin{proof}

Let us first prove the inequality. We note that (\ref{eqprs3}) is equivalent to
\begin{equation}\label{eqprs2}
 V(A,-C)\le V(A,C)+\sqrt{|A|\,|C|}.
\end{equation}
Assume  $A=A_1+A_2$ and 
$$
 V(A_1,-C)\le V(A_1,C)+\sqrt{|A_1|\,|C|} \mbox{ and } 
 V(A_2,-C)\le V(A_2,C)+\sqrt{|A_2|\,|C|}.
$$
Then
$$
V(A, -C)=V(A_1,-C)+V(A_2,-C) \le V(A,C)+\sqrt{|A_1|\,|C|}+\sqrt{|A_1|\,|C|}.
$$
Using that by Brunn-Minkowski inequality in the plane $\sqrt{|A_1|}+\sqrt{|A_1|} \le \sqrt{|A|}$ we get
$$
V(A, -C) \le V(A,C)+\sqrt{|A|\,|C|}.
$$
Thus,   to prove $(\ref{eqprs2})$ we may assume that both $A$ and $C$ are triangles. Indeed, any planar convex body  can be approximated by a polygon and any planar polygon can be written as a Minkowski sum of triangles \cite{Sch14:book} (here we will treat a segment as a degenerate triangle). If $A$ or $C$ is a segment, then the inequality becomes an equality. Thus, we may assume that $A$ and $C$ are not degenerate triangles. We notice that (\ref{eqprs2}) is invariant under dilation and shift of the convex body $C$, thus we may assume that $C \subset A$ and has common points with all three edges of $A$. Thus $V(A,C)=|A|$ (see, for example, \cite{ SZ16}). Finally, our goal is to show that for any triangles $A,C$, such that $C \subset A$ and $C$ touches all edges of $A$, we have
\begin{equation}\label{shadow}
V(A, -C) \le |A|+\sqrt{|A|\,|C|}.
\end{equation}
To prove the above inequality, one may use the technique of shadow systems (\cite{RS58:2} or \cite{Sch14:book}, Section 10.4).  In this particular case, the method can be applied directly.
Indeed, if $C=A$ then (\ref{shadow}) becomes an equality. Otherwise, let $C=\mbox{conv}\{c_1, c_2, c_3\}$ and one of the $c_i's$ is not a vertex of $A$. Assume, without loss of generality, that $c_1$ is not a vertex of $A$. Then, there exists a vertex $a_1$ of $A$ such that the segment $(a_1,c_1)$ does not intersect $C$. Let $c_{t}=tc_1 +(1-t)a_1$ for $t\in[t_0,t_1]$, where $[t_0,t_1]$ is the largest interval such that $c_{t}\in A$ and $c_{t}$ is not aligned with $c_2$ and $c_3$, for all $t \in (t_0,t_1)$. Notice that $c_{t_0}$ and $c_{t_1}$ are either vertices of $A$ or belong to the line containing $[c_2,c_3]$. Let $C_t=\mbox{conv}\{c_{t}, c_2, c_3\}$. Then
$$
V(A, -C_t)=V(-A, C_t)=\frac{1}{2}\sum_{i=1}^3 h_{C_t}(-u_i) |F_i|,
$$
where $u_i,$  $i=1,2,3$ is a normal vector to the edge $F_i$ of $A$. The function $h_{C_t}(-u_i)$ is convex in $t$ and thus the same is true for $V(A,-C_t)$.  We also notice that $|C_t|$ is an affine function of $t$ on $[t_0,t_1]$ and thus
$$
g(t) = V(A, -C_t) - \sqrt{|A|\,|C_t|}
$$
is a convex function of $t\in [t_0,t_1]$ and thus $\max_{t\in [t_0,t_1]}=\max\{g(t_0), g(t_1)\}.$ Thus the maximum of $g(t)$ achieved when either when $C_{t}$ becomes a segment (and the proof is complete in this case) or $c_{t}$ reaches a vertex of $A$. Then either $C=A$ or there is a vertex of $C$ which is not a vertex of $A$ and we repeat the procedure. 

Now let us consider the equality case. Assume 
\begin{equation}\label{equaltrian}
 V(A,-C) = V(A,C)+\sqrt{|A|\,|C|}.
\end{equation}
First, let us assume that  $A$ is not a triangle. Then $A$ is decomposable: it can be written as $A=A_1+A_2$ where $A_1$ is not homothetic to $A_2$. Then, applying (\ref{eqprs2}) we get
$$
 V(A_1,-C)\le V(A_1,C)+\sqrt{|A_1|\,|C|} \mbox{ and } 
 V(A_2,-C)\le V(A_2,C)+\sqrt{|A_2|\,|C|}.
$$
From (\ref{equaltrian}), we have equality in the above inequalities and 
$$
\sqrt{|A_1|\,|C|}+\sqrt{|A_2|\,|C|} =\sqrt{|A_1+A_2|\,|C|}.
$$
The above equality is only possible in two cases: first, when there is an equality in Brunn-Minkowski inequality, which would result $A_2$ to be homothetic to $A_1,$ and we assumed that this is not the case, and the second case,  when  $|C|=0,$ i.e. $C$ is the singleton or a segment.
If $C$ is not a triangle, then, the above discussion shows that $|A|=0$. 

Now, we assume that $A$ and $C$ are non degenerate triangles. Using homogeneity of equality  (\ref{equaltrian}) we may assume that the triangle $C$ touches all edges of $A$ and equality  (\ref{equaltrian})  becomes
\begin{equation}
 V(A,-C)- \sqrt{|A|\,|C|} = |A|.
\end{equation}
Assume towards the contradiction  that  $C \not = A$. Then, there is a vertex $c_1$ of the triangle $C=\conv(c_1,c_2,c_3)$ which is not a vertex of $A$. We reproduce the same shadow system $(C_t)_{t\in[t_0,t_1]}$ as in the proof of the inequality. We only need to prove that the function $g(t) = V(A, -C_t) - \sqrt{|A|\,|C_t|}$ is not constant on $[t_0,t_1]$. To do this, we prove that, among the two functions, $V(A, -C_t)$ and $\sqrt{|C_t|}$ at least one is not affine. There are two cases. If $|C_t|$ is not constant then $\sqrt{|C_t|}$ is strictly concave, thus $g$ is not a constant. If $|C_t|$ is constant then all vertices of $C$ are different from the vertices of $A$. Recall that $V(A, -C_t)=\frac{1}{2}\sum_{i=1}^3 h_{C_t}(-u_i)|F_i|$. Let $u_2, u_3$ be the normal of the edges of $A$ which do not contain $c_1$. Then it is easy to see that $h_{C_t}(-u_2)+h_{C_t}(-u_3)$ is not affine. Thus $V(A,-C_t)$ is not constant. This is a contradiction. 
\end{proof}

The inequality \eqref{eqprs3} is an intriguing improvement (in dimension 2) of Litvak's observation. To see this, observe that by the Brunn-Minkowski inequality,  
$|A+C|\geq (\sqrt{|A|}+\sqrt{|C|})^2
= |A|+|C|+2\sqrt{|A|\,|C|}
= (\sqrt{|A|}-\sqrt{|C|})^2+4\sqrt{|A|\,|C|}$,
and hence the right hand side of \eqref{eqprs3} is bounded by $\frac{3}{2}|A+C|$.
Thus, in dimension 2, since $c_2=1$, we obtain
\ben\begin{split}
|A|\,|A+B+C| &\le |A+B|\,|A+C| \\
&\le (|A-B|+2\sqrt{|A|\,|B|})\,|A+C|\\
&\le \bigg[\frac{3}{2}|A-B|-\frac{1}{2}(\sqrt{|A|}-\sqrt{|B|})^2\bigg] |A+C|,
\end{split}\een
which is an improvement of Theorem \ref{thm:ruzsa} in dimension 2.

Let us define the {\it additive asymmetry} of the pair $(A,C)$ by $$
\text{asym}(A,C)=\bigg| \frac{|A+C|}{|A-C|}-1\bigg|,
$$
and note that $\text{asym}(A,C)$ is trivially 0 if either $A$ or $C$ is symmetric. 
Observe that inequality \eqref{eqprs3}  may be rewritten as
\be\label{asym}
\text{asym}(A,C) \leq 2e^{-d(A,C)} ,
\ee
where $d(A,C)=\log \frac{|A-C|}{\sqrt{|A|\cdot |C|}}$ is the Euclidean analogue of the Ruzsa distance defined at the beginning of this section. 
One wonders if this inequality extends to dimension higher than 2.

Finally we note that comparison of cardinality of $A+C$ and $A-C$
has also been of interest in additive combinatorics (see, e.g., \cite{PF73, TV06:book}),
and there are also results comparing the entropies of sums and differences
of independent random variables in different abelian groups (see, e.g., \cite{MK10:isit, KM14, ALM17, LM19}).

\bibliographystyle{plain}
\bibliography{pustak}

\noindent Matthieu Fradelizi \\
LAMA, Univ Gustave Eiffel, Univ Paris Est Creteil, CNRS, F-77447 Marne-la-Vall\'ee,
France.\\
E-mail address: matthieu.fradelizi@univ-eiffel.fr 

\vspace{0.8cm}

\noindent Mokshay~Madiman\\
University of Delaware,
Department of Mathematical Sciences,
501 Ewing Hall, 
Newark, DE 19716, USA.
E-mail address: madiman@udel.edu

\vspace{0.8cm}

\noindent Artem Zvavitch \\
Department of Mathematical Sciences,
Kent State University,
Kent, OH 44242, USA,
E-mail address: azvavitc@kent.edu

\end{document}

%% file: paper_defns.tex
\usepackage{amsmath, amsthm}
\usepackage{amsfonts,amssymb}

\newcommand{\bc}{\begin{center}}
\newcommand{\ec}{\end{center}}
\newcommand{\bt}{\begin{tabular}}
\newcommand{\et}{\end{tabular}} 
\newcommand{\bea}{\begin{eqnarray}}
\newcommand{\eea}{\end{eqnarray}}
\newcommand{\bean}{\begin{eqnarray*}}
\newcommand{\eean}{\end{eqnarray*}}

\newcommand{\ba}{\begin{array}}
\newcommand{\ea}{\end{array}}

\def\be{\begin{eqnarray}}
\def\ee{\end{eqnarray}}
\def\ben{\begin{eqnarray*}}
\def\een{\end{eqnarray*}}

\newcommand{\ra} {\rightarrow}

\newcommand{\nth}{\frac{1}{n}}

\newcommand{\RL}{{\mathbb R}}

\newcommand{\calM}{\mbox{${\cal M}$}}

\newcommand{\Nat}{\mathbb{N}}

\def\elabel#1{\label{e:#1}}

\def\sq{$\Box$}

\def\qed{\ifmmode\sq\else{\unskip\nobreak\hfil
\penalty50\hskip1em\null\nobreak\hfil\sq
\parfillskip=0pt\finalhyphendemerits=0\endgraf}\fi\par\medbreak}

\newsavebox{\junk}
\savebox{\junk}[1.6mm]{\hbox{$|\!|\!|$}}

\def\til={{\widetilde =}}

 \def\eq#1/{(\ref{#1})}

\def\eq#1/{(\ref{e:#1})}

\newcommand{\beqn}[1]{\notes{#1}%
\begin{eqnarray} \elabel{#1}}

\newcommand{\eeqn}{\end{eqnarray} }

\newcommand{\beq}[1]{\notes{#1}%
\begin{equation}\elabel{#1}}

\newcommand{\eeq}{\end{equation}} 

\def\bdes{\begin{description}}
\def\edes{\end{description}}

\def\notes#1{}

\newcommand{\setS}{s}

\newcommand{\setT}{t}

\newcommand{\collS}{\mathcal{C}}

\def\phi{\varphi}

\def\bee{\begin{eqnarray*}}
\def\ene{\end{eqnarray*}}

\newcommand{\conv}{\mathrm{conv}}

\newcommand{\calA}{\mathcal{A}}
\newcommand{\calB}{\mathcal{B}}
\newcommand{\R}{\mathbb{R}}
\newcommand{\N}{\mathbb{N}}